\documentclass{amsart}
\usepackage{amssymb, amsthm, amsmath, amsfonts,amscd}
\usepackage{graphics}
\usepackage{hyperref}
\usepackage[all]{xy}
\usepackage{enumerate}
\usepackage[mathscr]{eucal}
\usepackage{bbm}
\usepackage{tikz}
\usetikzlibrary{decorations.pathreplacing,angles,quotes}

\providecommand{\U}[1]{\protect\rule{.1in}{.1in}}
\def\theenumi{\arabic{enumi}}

\def\theenumii{\alph{enumii}}
\def\p@enumii{\theenumi.}

\def\theenumiii{\arabic{enumiii}}
\def\p@enumiii{(\theenumi)(\theenumii)}

\def\p@enumiv{\p@enumiii.\theenumiii}
\theoremstyle{plain}
\newtheorem{theorem}{Theorem}[section]
\newtheorem{lemma}[theorem]{Lemma}
\newtheorem{proposition}[theorem]{Proposition}

\newtheorem{corollary}[theorem]{Corollary}

\numberwithin{equation}{section}

\theoremstyle{definition}

\newtheorem{definition}[theorem]{Definition}

\newtheorem{remark}[theorem]{Remark}

\newtheorem{thmab}{Theorem}

\renewenvironment{proof}[1][\proofname]{{\bfseries #1\\}}{\qed}

\DeclareMathOperator{\FI}{FI}
\DeclareMathOperator{\OI}{OI}
\DeclareMathOperator{\Mod}{-Mod}

\DeclareMathOperator{\Hom}{Hom}

\DeclareMathOperator{\op}{op}

\DeclareMathOperator{\Conf}{Conf}
\DeclareMathOperator{\DConf}{DConf}
\DeclareMathOperator{\sink}{sink}
\DeclareMathOperator{\rank}{rank}
\DeclareMathOperator{\hdim}{hdim}
\newcommand{\Sn}{\mathfrak{S}}
\newcommand{\as}{\text{*}}

\newcommand{\C}{{\mathcal{C}}}

\newcommand{\Z}{{\mathbb{Z}}}
\newcommand{\N}{\mathbb{N}}
\newcommand{\R}{\mathbb{R}}
\newcommand{\Q}{\mathbb{Q}}
\newcommand{\M}{\mathcal{M}}
\newcommand{\K}{\mathcal{K}}

\newcommand{\dt}{\bullet}
\newcommand{\arXiv}[1]{\href{http://arxiv.org/abs/#1}{\nolinkurl{arXiv:#1}}}
\newcommand{\arXivV}[2]{\href{http://arxiv.org/abs/#1}{\nolinkurl{arXiv:#1v#2}}}

\title{Configuration spaces of graphs with certain permitted collisions}

\author{Eric Ramos}
\address{Department of Mathematics, University of Wisconsin - Madison.}
\email{eramos@math.wisc.edu}

\thanks{The author was supported by NSF grant DMS-1502553.}

\begin{document}

\maketitle

\begin{abstract}
If $G$ is a graph with vertex set $V$, let $\Conf_n^{\sink}(G,V)$ be the space of $n$-tuples of points on $G$, which are only allowed to overlap on elements of $V$. We think of $\Conf_n^{\sink}(G,V)$ as a configuration space of points on $G$, where points are allowed to collide on vertices. In this paper, we attempt to understand these spaces from two separate, but closely related, perspectives. Using techniques of combinatorial topology we compute the fundamental groups and homology groups of $\Conf_n^{\sink}(G,V)$ in the case where $G$ is a tree. Next, we use techniques of asymptotic algebra to prove statements about $\Conf_n^{\sink}(G,V)$, for general graphs $G$, whenever $n$ is sufficiently large. It is proven that, for general graphs, the homology groups exhibit generalized representation stability in the sense of \cite{R2}.
\end{abstract}

\section{Introduction}

\subsection{Introductory remarks}
Let $Y$ be a topological space. The \textbf{$n$-strand configuration space on $Y$} is the space
\[
\Conf_n(Y) := \{(y_1,\ldots,y_n) \in Y^n \mid y_i \neq y_j\}
\]
For much of their history, configuration spaces have been studied in the case wherein $Y$ is a manifold of dimension at least 2. More recently, there have been various efforts made to understand configuration spaces of graphs (see \cite{G}\cite{A}\cite{KP}\cite{BF}\cite{CL}, for instance). In this context, a \textbf{graph} is a connected, compact, CW-complex of dimension 1. If $Y$ is a graph with at least one vertex of degree $\geq 3$, then it is a fact that the spaces $\Conf_n(Y)$ are always connected \cite{A}. In spite of this, there is a sense in which it is ``harder'' for points on a graph to move around each other than it would be in the case of higher dimensional manifolds. This fact manifests itself in many ways in the study of these spaces, and makes certain aspects of the theory much harder than the analogous aspects in the configuration spaces of manifolds.\\

One approach to mitigating this issue is by allowing for certain collisions. For instance, in \cite{Che} Chettih considers configuration spaces wherein you only disallow more than $k$ coordinates from coinciding, for some $k$. In this paper, we also consider configuration spaces of graphs with certain permitted collisions. Unlike the work of Chettih, where points are allowed to collide anywhere on $G$ so long as less than $k$ points are doing so, we consider spaces where points are only allowed to overlap on vertices of the graph. More formally, let $G$ be a graph with vertex set $V$. Then we consider the spaces,
\[
\Conf_n^{\sink}(G,V) := \{(x_1,\ldots,x_n) \in G^n \mid x_i \neq x_j \text{ or } x_i = x_j \in V\}.
\]
We refer to $\Conf_n^{\sink}(G,V)$ as the \textbf{$n$-strand sink configuration space of $G$}. These spaces were first considered by Chettih and L\"utgehetmann in \cite{CL}, although they only look at the cases of a circle, and a line segment. In this work, our goal will be to develop the theory of these spaces for all graphs.\\

\subsection{Topological aspects of $\Conf_n^{\sink}(G,V)$}

We will begin our study of sink configuration spaces from a more traditional topological perspective. In particular, we will initially spend time constructing a cellular model for the spaces $\Conf_n^{\sink}(G,V)$.\\

\begin{thmab}\label{fcm}
Let $G$ be a graph with vertex set $V$ and edge set $E$. Then the sink configuration space $\Conf_n^{\sink}(G,V)$ is homotopy equivalent to a CW complex of dimension at most $|E|$. If $G$ is a tree, then this CW complex is a cubical complex.\\
\end{thmab}

In the case of classical configuration spaces of graphs, Abrams has shown \cite{A} that $\Conf_n(G)$ is homotopy equivalent to a CW complex. Refining this, Ghrist \cite{G} showed that $\Conf_n(G)$ is homotopy equivalent to a CW complex whose dimension at most the number of vertices of $G$ of degree at least 3. Our result therefore represents one of the (many) similarities between sink configuration spaces and classical configuration spaces of graphs. One notable fact about both Theorem \ref{fcm} and the theorem of Ghrist is that the dimension of these spaces are bounded independent of $n$. While it is known that configuration spaces of manifolds can have cellular models (see \cite{FH}), these models do not usually have dimension which is independent of $n$.\\

Following this construction, we will be able to prove many topological facts about sink configuration spaces. For instance, this model will easily allow us to compute the Euler characteristic of $\Conf_n^{\sink}(G,V)$ (see Corollary \ref{eulerchar}) and prove that these spaces are $K(\pi,1)$ (see Theorem \ref{kpi1}). Perhaps more substantially, we will also be able to compute the fundamental group of $\Conf_n^{\sink}(G,V)$ whenever $G$ is a tree.\\

\begin{thmab} \label{treepi1}
Let $G$ be a tree with vertex set $V$ and edge set $E$. Then $\pi_1(\Conf_n^{\sink}(G,V))$ admits a presentation with $|E|\cdot((n-2)2^{n-1}+1)$ many generators, whose every relation is a commutator. If $G$ is a line segment and $|V| = 2$, then $\pi_1(\Conf_n^{\sink}(G,V))$ is free, whereas if $|V| = 3$ then $\pi_1(\Conf_n^{\sink}(G,V))$ is right-angled Artin.\\
\end{thmab}

This theorem was proven in the case wherein $G$ is a line segment with two vertices in \cite{CL}. Note that Theorem \ref{treepi1} is proven via an explicit induction argument and Van Kampen's Theorem. It will follow that one may always, in principle, compute these fundamental groups explicitly (see Theorem \ref{treepi}). At this time it is unclear whether $\pi_1(\Conf_n^{\sink}(G,V))$ is always right-angled Artin, as it is unclear whether one can choose a presentation wherein the relations are commutators of generators, as opposed to just commutators. This theorem is also significant, as it is the first evidence of the following fact: the spaces $\Conf_n^{\sink}(G,V)$ will, in general, depend heavily on the vertex and edge sets of $G$. In particular, $\Conf_n^{\sink}(G,V)$ is not a homeomorphism invariant of the graph $G$, although it is invariant under cellular homeomorphisms. This is in sharp contrast to the classical configuration spaces of graphs, which are unchanged by introducing vertices to $G$.\\

While the cellular model of Theorem \ref{fcm} is convenient for developing visual intuition for the sink configuration spaces (see Figure \ref{cellexample}), it is not quite refined enough to determine anything particularly significant about the homology groups of $\Conf_n^{\sink}(G,V)$. Towards this end, we will apply techniques from discrete Morse theory.\\

\begin{thmab}\label{sinkdmt}
There exists a discrete gradient vector field $X$ on the space $\Conf_n^{\sink}(G,V)$ for which the number of critical $i$-cells of $X$ is bounded by $p_i(n)(i+1)^n$, where $p_i$ is a polynomial of degree $i$.\\
\end{thmab}

For the reader unfamiliar with discrete Morse theory, Section \ref{dmt} provides a self-contained introduction to the subject. Put briefly, discrete Morse theory is a  means by which one can determine which cells of a CW complex are ``critical'' to the homotopy type of the space. Similar to classical Morse theory, one defines a map, called in this context a \textbf{discrete vector field}, which divides the cells of the space to three categories: redundant, collapsible, and critical. It is then shown that the space will deformation retract onto a complex with $m_i$ $i$-cells, where $m_i$ is the number of critical $i$-cells. Discrete Morse theory was originally developed by Foreman \cite{Fo1}, and has since seen tremendous application in a variety of ways throughout combinatorial topology. Farley and Sabalka, for instance, placed a discrete Morse structure on the so-called unordered configuration spaces of graphs in \cite{FS}. They then used this structure to prove a variety of theorems about the homology and fundamental groups of these unordered configuration spaces.\\

We will use the structure granted to us by Theorem \ref{sinkdmt} to prove the following.\\

\begin{thmab}\label{dmtmainthm}
Let $G$ be a tree with vertex set $V$ and edge set $E$. Then for all $n,i \geq 0$
\begin{enumerate}
\item $H_i(\Conf_n^{\sink}(G,V))$ is torsion free;
\item the homological dimension of $\Conf_n^{\sink}(G,V)$ is given by
\[
\hdim(\Conf_n^{\sink}(G,V)) = \min\{\lfloor \frac{n}{2} \rfloor, |E|\}
\]
\item the groups $H_i(\Conf_n^{\sink}(G,V))$ depend only on $i,n,$ and $|E|$.\\
\end{enumerate}
\end{thmab}

There are a few remarks we must make about the above theorem. Firstly, it has been shown by Chettih and L\"utgehetmann \cite{CL} that the homology groups of $\Conf_n(G)$ are always torsion-free, with no assumptions on $G$. While it is very possible that this is the case in the context of sink configuration spaces, we do not go beyond the case of trees in this paper. Secondly, the third part of Theorem \ref{dmtmainthm} is similar to Theorem \ref{treepi1} in that it asserts that certain topological invariants of $\Conf_n^{\sink}(G,V)$ depend only on simple combinatorial invariants of $G$ and its vertex set. We therefore stress that this simplicity seems entirely unique to trees and more basic graphs. Using the techniques of this paper, one can construct two graphs $G,G'$ with the same number of vertices and edges for which the homologies of $\Conf_n^{\sink}(G,V)$ and $\Conf_n^{\sink}(G',V')$ are not isomorphic. For instance, take $G$ to be the graph which looks like the capital letter 'Q' and take $G'$ to be the circle with two vertices.\\

Finally, one might immediately observe that Theorems \ref{treepi1} and \ref{dmtmainthm} are applications of our cellular models to the cases where $G$ is a tree. Similar to the classical theory of configuration spaces of graphs, it appears that studying $\Conf_n^{\sink}(G,V)$ for general graphs $G$ is quite difficult. We therefore will approach these cases from an entirely different perspective; that of asymptotic algebra.\\

\subsection{Sink configuration spaces and representation stability}

While it appears to be quite difficult to say anything about the homology groups of $\Conf_n^{\sink}(G,V)$ for general graphs, it is often times easier to deduce facts about these groups while allowing $n$ to vary. Many of the great recent advancements in the study of configuration spaces of manifolds have been achieved through the use of \textbf{representation stability} techniques (see, \cite{M}\cite{HR}\cite{CEF}\cite{MW} for a small sampling). The language of representation stability was first introduced by Church and Farb in \cite{CF} to provide a language for dealing with families of algebraic objects which were somehow consistently acted on by the symmetric groups. More precisely, let $\{X_n\}_{n \in \N}$ be a family of modules over a fixed commutative ring $k$ such that,
\begin{enumerate}
\item for every injection of sets $f:[m] \rightarrow [n]$ there is a map $f_\as:X_m \rightarrow X_n$;
\item if $f,g$ are two composable injections then $(f\circ g)_\as = f_\as \circ g_\as$;
\item there is a finite set $\{v_i\} \subseteq \bigoplus_n X_n$ such that for all $n$, every element $v \in X_n$ can be expressed as $v = \sum_{i} a_i(f_i)_\as(v_i)$, where $f_i$ is some injection of sets, and $a_i \in k$.\\
\end{enumerate}
Then we say that the family $\{X_n\}$ is \textbf{representation stable}. For instance, configuration spaces $\Conf_n(Y)$ admit forgetful maps on their points, which, in turn, induce maps on their cohomology groups $f_\as: H^i(\Conf_m(Y)) \rightarrow H^i(\Conf_n(Y))$ for every injection of sets $f:[m] \hookrightarrow [n]$. It was proven by Church \cite[Theorem 1]{Chu}, and later reproven by Church, Ellenberg and Farb \cite[6.2.1]{CEF}, that the collection $H^i(\Conf_\dt(Y))$ is representation stable, assuming $Y$ is an orientable manifold of dimension at least 2. Since then, there has been a tremendous boom of asymptotically flavored arguments and results in the literature.\\

Coming back to the configuration spaces of graphs, it can be seen that their cohomologies cannot possibly be representation stable. As before, this issue seems to stem from the fact that it is much harder for points on a graph to move past one another. In one extreme case, for example, the $n$-strand configuration space of a line segment is homotopy equivalent to $n!$ disjoint points. For more complicated graphs, their configuration spaces are connected, but this doesn't prevent $n!$ type growth from appearing in the higher homology groups (see \cite[Proposition 4.1]{G}, for example). Despite this fact, recent work of L\"utgehetmann \cite{L1}\cite{L2}, as well as the author \cite{R1}, have made progress in treating the configuration spaces of graphs from a more asymptotic viewpoint. While it is not the case that the homology of sink configuration spaces of graphs are representation stable in the sense of Church and Farb, one of the main results of this paper is that they do exhibit generalized representation stability in the sense of \cite{R2}.\\

While we save the technical aspects of what this means for the body of the paper (see Section \ref{fiddef}), we can at least give an intuition now. Given an injection of sets $f:[m] \hookrightarrow [n]$ we would like to obtain a map of spaces $f_\as:\Conf_n^{\sink}(G,V) \rightarrow \Conf_n^{\sink}(G,V)$. Given no extra data, there doesn't seem to be any obvious way to do this. However, if we are also given a map $g$ from the compliment of the image of $f$ in $[n]$ to the set of vertices $V$, then we can define such a map. Namely, the pair $(f,g)$ induces a map $(f,g)_\as:\Conf_m^{\sink}(G,V) \rightarrow \Conf_n^{\sink}(G,V)$, where one uses $f$ to push forward the coordinates that are given, while using $g$ to determine to which vertices the newly introduced points should be sent. That is,
\[
(f,g)_\as(x_1,\ldots,x_m)_j = \begin{cases} x_{f^{-1}(j)} &\text{ if $j$ is in the image of $f$}\\ g(j) &\text{ otherwise.}\end{cases}
\]
This gives the homology groups $H_i(\Conf_\dt^{\sink}(G,V))$ what \cite{R2} calls generalized representation stability. One major consequence of this stability is the following.\\

\begin{thmab}\label{sinkexpstab}
The modules $H_i(\Conf_n^{\sink}(G,V))$ enjoy the following properties:
\begin{enumerate}
\item there exist polynomials $p_1^{(i)},\ldots,p_{\delta_i}^{(i)} \in \Q[x]$ such that for all $n \gg 0$
\[
\rank_\Z(H_i(\Conf_n^{\sink}(G,V))) = p_1^{(i)}(n) + p_2^{(i)}(n)2^n + \ldots + p_{\delta_i}^{(i)}(n)\delta_i^n
\]
where
\[
\delta_i = \min\{i+1, |V|\}
\]
\item the polynomial $p_{|V|}^{(|E|)}$ is non-zero, and if $G$ is a tree then $p_j^{(i)}$ is non-zero for all $0 \leq i \leq |E|$ and $j \leq i+1$;
\item there is a finite integer $e_i^{(G,V,E)}$, independent of $n$, for which the exponent of $H_i(\Conf_n^{\sink}(G,V))$ is at most $e_i^{(G,V,E)}$.\\
\end{enumerate}
\end{thmab}

We note that the second part of Theorem \ref{sinkexpstab} implies that the homological dimension of $\Conf_n^{\sink}(G,V)$ is equal to $|E|$ for $n \gg 0$. Indeed, Theorem \ref{fcm} implies that $|E|$ is an upper bound on the homological dimension, while the second part of Theorem \ref{sinkexpstab} implies that this bound is eventually realized. At this time it is unclear for what values of $n$ this homological dimension is realized, which is in contrast to the explicit bounds of Theorem \ref{dmtmainthm}.\\

For the reader more knowledge about generalized representation stability in the sense of \cite{R2}, it is interesting to note that if $G$ is a tree with vertex set $V$, then the homology groups $H_{|E|}(\Conf_\dt^{\sink}(G,V))$ exhibit the full range of possible growth. That is to say, We may write
\[
\rank_\Z(H_i(\Conf_n^{\sink}(G,V)) = p_1^{(|E|)}(n) + p_2^{(|E|)}(n)2^n + \ldots + p_{|V|}^{(|E|)}(n)|V|^n,
\]
where none of the polynomials $p_j^{(|E|)}$ are zero. To the knowledge of the author, sink configuration spaces are one of the first naturally arising examples of such behavior. Finally, note that computations done by the author suggest that the polynomials $p_j^{(i)}$ should be non-zero for all graphs $G$, and all $0 \leq i \leq |E|$ with $j \leq \delta_i$.\\

If we restrict ourselves to rational homology, the groups $H_i(\Conf_\dt^{\sink}(G,V);\Q)$ form a collection of complex symmetric group representations. As the name representation stability suggests, we should be able to say something about these representations as $n$ is allowed to vary. To this end we have the following.\\

\begin{thmab}\label{sinkrepstab}
Assume that $|V| = d$. For any $i$, The complex $\Sn_n$-representations $H_i(\Conf_n^{\sink}(G,V);\Q)$ enjoy the following properties:
\begin{enumerate}
\item for any partition $\lambda$, and any integers $n_1 \geq \ldots \geq n_d \geq |\lambda| + \lambda_1$, let $c_{\lambda,n_1,\ldots,n_d}$ be the multiplicity of $S(\lambda)_{n_1,\ldots,n_d}$ in $H_i(\Conf_{\sum_i n_i - (d-1)|\lambda|}^{\sink}(G,V);\Q)$ (see Definition \ref{padpar}). Then the quantity $c_{\lambda,n_1+l,\ldots,n_d+l}$ is independent of $l$ for $l \gg 0$;
\item if $c_{\lambda,n}$ is the multiplicity of $S(\lambda)_n$ in $H_i(\Conf_n^{\sink}(G,V))$, then there exists a polynomial $p(x) \in \Q[x]$ of degree $\leq d-1$ such that for all $n \gg 0$, $c_{\lambda,n} = p(n)$;
\item there is a finite integer $b_i^{(G,V,E)}$, independent of $n$, such that for any irreducible representation $S^{\lambda}$ appearing in $H_i(\Conf_n^{\sink}(G,V))$, $\lambda$ has at most $b_i^{(G,V,E)}$ boxes below its $d$-th row.\\
\end{enumerate}
\end{thmab}

\section*{Acknowledgments}
The author would like to send thanks to Jordan Ellenberg, John Wiltshire-Gordon, Jennifer Wilson and Graham White for various useful conversations during the writing and conception of this work. The author would also like to send very special thanks to Steven Sam for his support during the initial stages of this work. Finally, the author would like to acknowledge the generous support of the NSF via the grant DMS-1502553.\\

\section{Preliminary notions}

\subsection{Graphs and sink configuration spaces}

\begin{definition}
A \textbf{graph} $(G,V,E)$ is a triple of a compact connected CW-complex of dimension one $G$, along with the set of its 0-cells $V$ and its 1-cells $E$. An \textbf{isomorphism} of graphs $(G,V,E) \cong (G',V',E')$ is a cellular homeomorphism $\phi:G \rightarrow G'$. A \textbf{tree} is a graph $(G,V,E)$ for which $G$ is contractible. A \textbf{loop} in a graph $(G,V,E)$ is a 1-cell $e \in E$ whose endpoints are identified in $G$.\\
\end{definition}

Given a graph $(G,V,E)$, we will often slightly abuse nomenclature and refer to $G$ as being a graph. While this distinction is unimportant in many instances, it turns out to be quite relevant in the study of sink configuration spaces (see Definition \ref{sinkdef}). In the coming sections we will find that given a 1-dimensional CW-complex G, two non-isomorphic cellular models of $G$ can yield sink configuration spaces with drastically different topologies.\\

\begin{definition}
Let $(G,V,E)$ be a graph. The \text{$n$-th configuration space of $(G,V,E)$} is the topological space
\[
\Conf_n(G) = \{(x_1,\ldots,x_n) \in G^n \mid x_i \neq x_j \text{ if $i \neq j$}\}.
\]
\text{}\\
\end{definition}

Configuration spaces of graphs have seen a surge of interest due to their connections with robotics \cite{G}\cite{Far}, as well as their many interesting theoretical properties (see \cite{KP}\cite{A}\cite{FS}\cite{BF}\cite{CL}, for a sampling). There has also been a recent push in the literature to understand these spaces from the perspective of representation stability \cite{L1}\cite{L2}\cite{R1}. In this paper we will use the existing literature on the configuration spaces of graphs as a guiding tool for developing a theoretical framework for studying what we call sink configuration spaces.\\

\begin{definition}\label{sinkdef}
Let $(G,V,E)$ be a graph. The \textbf{$n$-th sink configuration space of $(G,V,E)$} is the topological space
\[
\Conf_n^{\sink}(G,V) = \{(x_1,\ldots,x_n) \in G^n \mid x_i \neq x_j \text{ or } x_i = x_j \in V\}.
\]
\text{}\\
\end{definition}

The cases $([0,1],\{0,1\},\{(0,1)\})$ and $(S^1,\{0\},\{(0,1)\})$ were studied in \cite{CL}. The work \cite{CL} is, to the knowledge of the author, also the first paper where these spaces were considered. Note that, unlike with normal configuration spaces, $\Conf^{\sink}_n(G,V)$ depends on the choice of cellular model of $G$ (see Theorem \ref{homdim}). For this reason we indicate the vertex set in our notation for the sink configuration spaces.\\

Much of the difficulty with studying configuration spaces of graphs comes from the fact that it is difficult for points on a graph to move past one another. Sink configuration spaces attempt to mitigate this problem by allowing points to collide with one another on vertices. We will find that, in many ways, this change makes sink configuration spaces easier to study.\\

\subsection{$\FI_V$-modules}\label{fiddef}

For the remainder of this section, we fix a graph $(G,V,E)$.\\

One way we will approach the study of sink configuration spaces is through the language of asymptotic algebra. We will find that Church-Farb representation stability theory \cite{CF} is not quite general enough to cover these spaces (see, for instance, Theorem \ref{homdim}). Instead, we will apply techniques of more general representation stability theories.\\

\begin{definition}\label{freemod}
The category $\FI_V$ is that whose objects are finite sets, and whose morphisms are pairs $(f,g):S \rightarrow T$ of an injection of sets $f:S \hookrightarrow T$ with a map of sets $g:T-f(S) \rightarrow V$. If $(f,g)$ and $(f',g')$ are two composable morphisms, then we set
\[
(f,g) \circ (f',g') = (f \circ f',h),
\]
where 
\[
h(x) = \begin{cases} g(x) &\text{ if $x \notin \text{im}f$}\\ g'(f^{-1}(x)) &\text{ otherwise.} \end{cases}
\] 
Note that if $|V| = d$, then $\FI_V$ is equivalent to the category $\FI_d$, which can be seen discussed in \cite{R2}, \cite{SS}, \cite{SS2}, \cite{Sa1} and \cite{Sa2}. The work \cite{SS2} approaches this category from a very different, but equivalent, perspective. Given a commutative ring $k$, an \textbf{$\FI_V$-module over $k$} is a functor $W:\FI_V \rightarrow k\Mod$. A \textbf{morphism of $\FI_V$} modules is a natural transformation of functors. The category of $\FI_V$-modules is denoted $\FI_V\Mod$. This is an abelian category, with abelian operations defined point-wise.\\

We will often use $W_S$ to denote $W(S)$, and $(f,g)_\as$ to denote $W(f,g)$. Given a finite set $S$, we define the \textbf{free $\FI_V$-module on $S$} to be the unique module $M(S)$ which satisfies
\[
\Hom_{\FI_V\Mod}(M(S),W) = W_S
\]
for every $\FI_V$-module $W$. Equivalently, $M(S)$ is the $\FI_V$-module which is defined on points by
\[
M(S)_T = k[\Hom_{\FI_V}(S,T)],
\]
and for which morphisms act by composition. We say that an $\FI_V$-module $W$ is \textbf{finitely generated} if it is surjected onto by a direct sum of free modules. That is, if there is a finite collection $\{v_i\} \subseteq \sqcup_{S} W_S$ which generates the whole of $W$.\\
\end{definition}

We note that the full subcategory of $\FI_V$ whose objects are the sets $[n] = \{1,\ldots,n\}$ is naturally equivalent to $\FI_V$. For this reason we will often restrict ourselves to working with this subcategory. Also observe that if $|V| = 1$, then $\FI_V$ is naturally equivalent to the category $\FI$ of finite sets and injections. $\FI$-modules were introduced by Church, Ellenberg, and Farb in their seminal work \cite{CEF}, and can be seen as one of the primary driving forces behind recent interest in so-called asymptotic algebra.\\

In this work we will be primarily interested in studying finitely generated $\FI_V$-modules. To this end, we have the following critical theorem.\\

\begin{theorem}[\cite{SS}, Theorems 7.1.2 and 7.1.5]\label{noeth}
Assume that $|V| = d$. If $W$ is a finitely generated $\FI_V$-module over a Noetherian ring $k$, then all submodules of $W$ are also finitely generated. Moreover, if $k$ is a field, then there exists polynomials $p_1^W,\ldots,p_d^W \in \Q[x]$ such that
\[
\dim_k(W_n) = p_1^W(n) + p_2^W(n)2^n + \ldots + p_d^W(n)d^n
\]
for all $n\gg 0$.
\end{theorem}

Prior to the provided source, the Noetherian property in the above theorem was proven for $\FI_d$-modules over a field of characteristic 0 in \cite[Theorem 2.3]{Sn}. The statement on dimensional stability was proven in the case where $k$ is a field of characteristic 0 as \cite[Theorem 3.1]{Sn}. The specific case of $\FI$-modules was also handled by Church, Ellenberg, Farb, and Nagpal in \cite{CEF} and \cite{CEFN}.\\

Specializing to the case where $k$ is a field of characteristic 0, we see that finitely generated $\FI_V$-modules can be examined from the perspective of the complex representation theory of the symmetric groups. Before we state our results, we first recall some notation.\\

\begin{definition}\label{padpar}
Recall that the irreducible representations of the symmetric group $\Sn_n$ over a field of characteristic 0 are in bijection with partitions $\lambda = (\lambda_1,\ldots,\lambda_h)$ of $n$. We will write $S^\lambda$ to denote this irreducible representation. If $\lambda$ is a partition of $n$ and $n_r \geq n_{d-1} \geq \ldots \geq n_1 \geq \lambda_1 + n$ is a sequence of integers, we define the \textbf{$r$-padded partition} $\lambda[n_1,\ldots,n_r] := (n_1-n,\ldots,n_r-n,\lambda_1,\ldots,\lambda_h)$. We will use $S(\lambda)_{n_1,\ldots,n_r} := S^{\lambda[n_1,\ldots,n_r]}$.\\
\end{definition}

\begin{theorem}[\cite{R2}, Theorems A and B Corollary 3.7]\label{genrepstab}
Assume that $|V| = d$. Let W be an $\FI_V$-module over a field $k$ of characteristic 0, and write $\phi^i_n:W_n \rightarrow W_{n+1}$ for the map induced by the pair of the standard inclusion $[n] \hookrightarrow [n+1]$ and the color $i$. Then $W$ is finitely generated if and only if $W_n$ is finite dimensional for all $n \geq 0$, and for all $n \gg 0$:
\begin{enumerate}
\item $\cap_i \ker\phi^i_n = \{0\}$;
\item $\sum_i \phi_n^i(W_n)$ spans $W_{n+1}$ as an $\Sn_{n+1}$-representation;
\item for any partition $\lambda$, and any integers $n_1 \geq \ldots \geq n_d \geq |\lambda| + \lambda_1$, let $c_{\lambda,n_1,\ldots,n_d}$ be the multiplicity of $S(\lambda)_{n_1,\ldots,n_d}$ in $W_{\sum_i n_i - (d-1)|\lambda|}$. Then the quantity $c_{\lambda,n_1+l,\ldots,n_d+l}$ is independent of $l$ for $l \gg 0$.\\
\end{enumerate}
Moreover, if $W$ is finitely generated, and $c_{\lambda,n}$ is the multiplicity of $S(\lambda)_n$ in $W_n$, then there exists a polynomial $p(x) \in \Q[x]$ of degree $\leq d-1$ such that for all $n \gg 0$, $c_{\lambda,n} = p(n)$. Finally, if $W$ is finitely generated then there exists an integer $b_W$, independent of $n$, such that if $S^\lambda$ is an irreducible representation appearing in $W_n$, then $\lambda$ has at most $b_W$ boxes below its $d$-th row.\\
\end{theorem}

Having stated Theorems \ref{noeth} and \ref{genrepstab}, one sees that Theorems \ref{sinkexpstab} and \ref{sinkrepstab} will follow from the modules $H_q(\Conf^{\sink}_n(G,V);\Q)$ forming a finitely generated $\FI_V$-module. While we will need more machinery before we can prove finite generation, we can at least present this $\FI_V$-module structure now.\\

\begin{definition}\label{fivaction}
Let $(f,g):[m] \rightarrow [n]$ be a morphism in $\FI_V$. Then we define an induced morphism
\[
(f,g):\Conf^{\sink}_m(G,V) \rightarrow \Conf^{\sink}_n(G,V), (f,g)(x_1,\ldots,x_m)_i = \begin{cases} x_{f^{-1}(i)} & \text{ if $i \in f([m])$}\\ g(i) & \text{ otherwise.}\end{cases}
\]
These assignments define a functor $Conf_\dt(G,V)$ from $\FI_V$ to the category of topological spaces. For any fixed $q$, composition with the homology functor $H_q(\dt)$ therefore defines an $\FI_V$-module over $\Z$
\[
H_q(\Conf^{\sink}_\dt(G,V)):\FI_V \rightarrow \Z\Mod.
\]
\text{}\\
\end{definition}

We see that allowing collisions to happen at vertices enables us to introduce points into the sink configuration space in a continuous way. This ability seems to be quite rare in the study of configuration spaces in general. Notable cases in the classical setting where analogous maps exist are the configuration spaces of manifolds with boundary \cite{CEF}\cite{MW}, and manifolds with an everywhere non-vanishing vector field \cite{EW-G}.\\

\begin{remark}
One should note that given an injection $f:[n] \hookrightarrow [m]$, there is a forgetful map $f:\Conf^{\sink}_m(G,V) \rightarrow \Conf^{\sink}_n(G,V)$. Therefore the $\FI_V$-module $H_q(\Conf^{\sink}_\dt(G,V))$ actually carries the structure of an $(\FI_V \times \FI^{\op})$-module. This extra structure imposes some fairly rigid restrictions on the module $H_q(\Conf^{\sink}_\dt(G,V))$. For instance, it can be shown that this extra structure implies the dimension formula of Theorem \ref{sinkexpstab} will hold for all $n$, instead of for $n$ sufficiently large. In the case of trees this follows from Theorem \ref{recursive}. Showing that this is the case for general $(\FI_V \times \FI^{\op})$-modules is fairly technical, and requires results from \cite{SS2}. As we do not prove this strengthening of Theorem \ref{sinkexpstab} in this work, the statement of that theorem will remain unchanged.\\
\end{remark}

\subsection{Discrete Morse theory}\label{dmt}

The first step in proving that the $\FI_V$-module $H_q(\Conf^{\sink}_\dt(G,V))$ is finitely generated, as well as the first step in proving many topological results about the spaces $\Conf^{\sink}_n(G,V)$, involves placing a cellular model on $\Conf^{\sink}_n(G,V)$. After this is accomplished, we will refine the model using discrete Morse theory. Discrete Morse theory was introduced by Forman as a means of applying classical Morse theory techniques to the study of CW-complexes \cite{Fo1}\cite{Fo2}. The exposition that follows is based on the exposition in \cite{R1}, which was inspired by the exposition in the works \cite{FS}\cite{Fo2}\cite{KP}.\\

\begin{definition}\label{dmtmaindefs}
Let $Y$ be a CW complex. A \textbf{cell} of $Y$ will always refer to an open cell in $Y$. Given a cell $\sigma$ of dimension $i$, we will often write $\sigma^{(i)}$ to indicate that $\sigma$ has dimension $i$. We will write $\K$ to denote the set of cells of $Y$ and $\K_i$ to denote the set of $i$-cells of $Y$.\\

A cell $\tau^{(i)} \subseteq \overline{\sigma^{i+1}}$ is said to be a \textbf{regular face} of a cell $\sigma^{(i+1)}$ if, given a characteristic map $\Phi_\sigma:D^{i+1} \rightarrow Y$ for $\sigma^{(i+1)}$, $\Phi_{\sigma}^{-1}(\tau)$ is a closed ball, and the map $\Phi_{\sigma}|_{\Phi_\sigma^{-1}(\tau)}$ is a homeomorphism.\\

A \textbf{discrete vector field} $X$ on $Y$ is a collection of partially defined functions $X_i:\K_i \rightarrow \K_{i+1}$ satisfying the following three conditions for each $i$:
\begin{enumerate}
\item $X_i$ is injective;
\item the image of $X_i$ is disjoint from the domain of $X_{i+1}$;
\item for any $\sigma^{(i)}$ in the domain of $X_i$, $\sigma^{(i)}$ is a regular face of $X_i(\sigma^{(i)})$.\\
\end{enumerate}

Given a CW complex $Y$ equipped with a discrete vector field $X$, a \textbf{cellular path} between two cells $\alpha^{(i)}$ and $\beta^{(i)}$ is a finite sequence of $i$-cells
\[
\alpha^{(i)} = \alpha_0^{(i)}, \alpha_1^{(i)}, \ldots, \alpha_{l-1}^{(i)}, \alpha_l^{(i)} = \beta^{(i)}
\]
such that $\alpha_{j+1}^{(i)}$ is a face of $X_i(\alpha_{j}^{(i)})$. We say that the path is \textbf{closed} if $\alpha^{(i)} = \beta^{(i)}$, and we say it is \textbf{trivial} if $\alpha_j^{(i)} = \alpha_{k}^{(i)}$ for all $j,k$.\\

A discrete vector field $V$ is said to be a \textbf{discrete gradient vector field} if it admits no non-trivial closed cellular paths.\\

If $X$ is a discrete gradient vector field on a CW complex $X$, then we call a cell $\sigma$ of $X$ \textbf{redundant} if $\sigma$ is in the domain of $X_i$ for some $i$, \textbf{collapsible} if it is in the image of $X_i$ for some $i$, and \textbf{critical} otherwise.\\
\end{definition}

To get a better intuition for the above definitions, one might also consider the following (equivalent) formulation. Let $f:\K \rightarrow \R$ be a map such that for all $\sigma \in \K_i$,
\begin{enumerate}
\item $|\{\tau \in \K_{i+1} \mid \sigma \subseteq \overline{\tau} \text{ and } f(\sigma) \geq f(\tau)\}| \leq 1$; \label{dm1}
\item $|\{\tau \in \K_{i-1} \mid \tau \subseteq \overline{\sigma} \text{ and } f(\sigma) \leq f(\tau)\}| \leq 1$. \label{dm2}
\end{enumerate}
One calls $f$ a discrete Morse function, and it can be thought of as being analogous to a Morse function in the classical sense. Cells for which \ref{dm1} is an equality are redundant, for which \ref{dm2} is an equality are collapsible, and for which both are not equalities are critical. It is a fact \cite{Fo1} that these are actually the only three cases. One immediately notes that according to this formulation, the classification of cells does not fully depend on the explicit values of the function $f$. In other words, postcomposition with any monotone function on $\R$ does not change the resulting classification of cells. One then constructs a discrete vector field by assigning a cell $\sigma \in \K_i$ to the (unique) cell $\tau \in \K_{i+1}$ containing $\sigma$ for which $f(\sigma) \geq f(\tau)$. The fact that this discrete vector field is actually a discrete gradient vector field, i.e. the fact that it admits no non-trivial closed cellular paths, is exhibited in \cite{Fo1} and \cite{Fo2}.\\

Just as with classical Morse theory, it is the case that the critical cells of a discrete gradient vector field $X$ on a space $Y$ determine the whole of $Y$.\\

\begin{proposition}[\cite{FS} Proposition 2.2, \cite{Fo1} Theorem 3.4]\label{VFdecomp}
Let $Y$ be a CW complex equipped with a discrete gradient vector field $X$. Consider the filtration
\[
\emptyset = Y_0'' \subseteq Y_0' \subseteq Y_1'' \subseteq Y_1' \subseteq \ldots \subseteq Y_n'' \subseteq Y_n' \subseteq \ldots
\]
where $Y_i'$ is the $i$-skeleton of $Y$ with the redundant cells removed, and $Y_i''$ is the $i$-skeleton of $Y$ with both the redundant and critical cells removed. Then:
\begin{enumerate}
\item For any $i$, $Y'_i$ is obtained from $Y''_i$ by attaching $m_i$ $i$-cells to $Y_n''$ along their boundaries, where $m_i$ is the number of critical $i$-cells of the discrete gradient vector field $X$.
\item For any $i$, $Y_{i+1}''$ deformation retracts onto $Y_i'$.\\
\end{enumerate}
\end{proposition}

The above proposition leads to one notable corollary, which we record now.\\

\begin{corollary}[\cite{FS} Proposition 2.3, \cite{Fo1} Corollary 3.5]\label{critdecomp}
Let $Y$ be a CW complex equipped with a discrete gradient vector field $X$. Then $Y$ is homotopy equivalent to a CW complex with precisely $m_i$ $i$-cells for each $i$, where $m_i$ is the number of critical $i$-cells of $X$.\\
\end{corollary}

Discrete Morse theory can therefore be thought of as a means of refining a cellular complex by removing and otherwise collapsing cells which are homotopically inessential. Just as is the case with traditional Morse theory, the decomposition of the space $Y$ given by Corollary \ref{critdecomp} can be used to compute the homology groups of $Y$. For simplicity, we will state the construction for cubical complexes, although the general case is similar. In the next section we will see that if $(G,V,E)$ is loopless, then $\Conf^{\sink}_n(G,V)$ is homotopy equivalent to a cubical complex, so this case is particularly relevant for what follows.\\

\begin{definition}\label{morsediff}
Let $Y$ be a cubical complex, and write $C_i(Y)$ for the free abelian group of $i$-cells of $Y$. If $c \cong I_1 \times \ldots \times I_i$ is an $i$-cell of $Y$, with $I_j \cong [0,1]$, then we define
\[
c^j_{\tau} = I_1 \times \ldots \times I_{j-1} \times \{0\} \times I_{j+1} \times \ldots \times I_i, \hspace{.5cm} c^j_{\iota} = I_1 \times \ldots \times I_{j-1} \times \{1\} \times I_{j+1} \times \ldots \times I_i.
\]
This allows us to define a boundary morphism
\[
\partial: C_i(Y) \rightarrow C_{i-1}(Y)
\]
given by
\[
\partial(c) := \sum_j c^j_\iota - c^j_\tau,
\]
turning $C_\dt(Y)$ into a chain complex. It is a well known fact that the homology of this chain complex is the usual homology of the space $Y$.\\

Further assume that $Y$ is equipped with a discrete gradient vector field $X$. Then we have a map $R:C_i(Y) \rightarrow C_i(Y)$ defined by
\[
R(c) = \begin{cases} 0 &\text{ if $c$ is collapsible}\\ c &\text{ if $c$ is critical}\\ \pm \partial(X_i(c)) + c &\text{ otherwise,}\end{cases}
\]
where the sign of $\partial(X_i(c))$ in the above definition is chosen so that $c$ has a negative coefficient. The property that $X$ has no non-trivial closed paths implies that $R^m(c) = R^{m+1}(c)$ for all $m \gg 0$ and all $i$-cells $c$ \cite{Fo1}. We set $R^\infty(c)$ to be this stable value.\\

For each $i$, let $\M_i$ denote the free abelian group with basis indexed by the critical $i$-cells of $X$. Then the \textbf{Morse complex} associated to $X$ is defined to be
\[
\M_\dt : \ldots \rightarrow \M_n \rightarrow \ldots \rightarrow \M_1 \stackrel{\widetilde{\partial}}\rightarrow \M_0 \rightarrow 0,
\]
where boundary map $\widetilde{\partial}$ is given by
\[
\widetilde{\partial}(c) := R^\infty(\partial(c))
\]
The map $\widetilde{\partial}$ is known as the \textbf{Morse differential}.\\
\end{definition}

\begin{theorem}[Theorem 8.2 \cite{Fo1}, Theorem 7.3 \cite{Fo2}]\label{morsediffcomp}
For all $i$ there are isomorphisms,
\[
H_i(X) \cong H_i(\M_\dt)
\]
\text{}\\
\end{theorem}

Farley and Sabalka introduced a discrete gradient vector field for the configuration spaces of graphs in \cite{FS}. The discrete gradient vector field that we construct for the spaces $\Conf^{\sink}_n(G,V)$ will be heavily inspired by that work (see Definition \ref{sinkdgvf}).\\

\section{Cellular models for $\Conf_n(G,V)$}

For the remainder of this section, we fix a graph $(G,V,E)$

\subsection{The first model: $\DConf^{\sink}_n(G,V)$} \label{dmtgraph}

The first cellular model we place on the sink configuration spaces of $(G,V,E)$ was first considered by Chettih and L\"utgehetmann in \cite{CL}. In that work, the model is explicitly constructed in the cases of $([0,1],\{0,1\},\{(0,1)\})$ and $(S^1,\{0\},\{(0,1)\})$. We will see that very little work is necessary to extend these two cases to arbitrary graphs.\\

\begin{definition}
We define the \textbf{discrete $n$ strand sink configuration space of (G,V,E)} as the subcomplex of $G^n$
\[
\DConf^{\sink}_n(G,V) := \cup \{\sigma_1 \times \ldots \times \sigma_n\}
\]
where the union is over all cells $\sigma_1 \times \ldots \times \sigma_n \subseteq G^n$ such that $\sigma_i \neq \sigma_j$ whenever $\sigma_i$ and $\sigma_j$ are both 1-cells of $G$.\\

Given a cell $\sigma = \sigma_1 \times \ldots \times \sigma_n$, if $\sigma_j$ is an edge (or vertex) of $G$ then we call $j$ the \textbf{index of a coordinate associated to $\sigma_j$}. We will also sometimes say that the coordinate $j$ \textbf{occupies} the edge (or vertex) $\sigma_j$. If $\sigma_j$ is an edge (or vertex) of $G$, then we say that $\sigma$ \textbf{contains} the edge (or vertex) $\sigma_j$.\\
\end{definition}

One should think of $\DConf^{\sink}_n(G,V)$ as $n$-tuples of points of $G$ such that every point is either located on a vertex of $G$, or alone on the interior of an edge of $G$. It is useful to visualize the cells of $\DConf^{\sink}_n(G,V)$ on a drawing of the graph $(G,V,E)$. For instance, if $(G,V,E)$ is the graph which looks like the letter ``Y,'' with the minimal number of vertices, then Figure \ref{cellexample} illustrates a 1-cell of $\DConf_6^{\sink}(G,V)$.\\

As an interesting case to consider, let $(G,V,E)$ be the interval with two vertices. Then the space $\DConf_n^{\sink}(G,V)$ is seen to be the 1-skeleton of the $n$-dimensional hypercube.\\

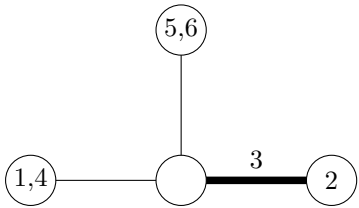
\begin{figure}
\begin{tikzpicture}

\draw [line width=1mm] (2,0) -- node[above] {3} (4,0);

\tikzstyle{every node}=[draw,circle,fill=white,minimum size=19pt,
                            inner sep=0pt]
    \draw (0,0) node (0)  {1,4}
					--(2,0) node (2) {}
					--(2,2) node (4) {5,6}
    			(2) -- (4,0) node (7) {2};
		 
\end{tikzpicture}
\caption{A visualization of a 1-cell of $\DConf_6^{\sink}(G,V)$. In this visualization, a vertex is filled in by the indices of the coordinates which occupy it. The interior of an edge is made bold if it is being occupied by a coordinate, and this (unique) coordinate is written above.}\label{cellexample}
\end{figure}

Note that $\DConf^{\sink}_n(G,V)$ also naturally embeds into the usual sink configuration space. The main result of this section is the following.\\

\begin{theorem}[\cite{CL}, Proposition 2.2]
The inclusion $\iota_n:\DConf^{\sink}_n(G,V) \hookrightarrow \Conf^{\sink}_n(G,V)$ is a homotopy equivalence.
\end{theorem}

\begin{proof}
As previously stated, the cited source proves this theorem in the cases of the interval and the circle. Their proof is based on an intermediate cellular model, which was first considered in \cite{L1}. In both of these cases, the retraction which defines the homotopy inverse of $\iota_n$ does not move points which are already situated on a vertex. As a consequence, given a general graph $(G,V,E)$, we may define our retraction in steps, where at each step we only move points which are on the interior a given edge or loop. The resulting chain of deformation retractions will eventually yield our desired equivalence.\\
\end{proof}

\begin{remark}\label{dfiv}
Note that the spaces $\DConf^{\sink}_\dt(G,V)$ have a natural action by $\FI_V$, and that this action is compatible with the inclusions
\[
\iota_\dt:\DConf^{\sink}_\dt(G,V) \hookrightarrow \Conf^{\sink}_\dt(G,V)
\]
It follows that the $\FI_V$-module $H_q(\Conf^{\sink}_\dt(G,V))$ is isomorphic to the $\FI_V$-module $H_q(\DConf^{\sink}_\dt(G,V))$, and we may prove facts about the former by proving facts about the latter. This will be our strategy when treating these spaces from the perspective of asymptotic algebra.\\
\end{remark}

We can now examine some topological consequences of the cellular model for $\Conf_n^{\sink}(G,V)$. To begin, we have the following.\\

\begin{corollary}
The homological dimension of $\Conf_n^{\sink}(G,V)$ is at most $|E|$, for all $n$. That is, for all $i > |E|$ and all $n \geq 0$, $H_i(\Conf_n^{\sink}(G,V)) = 0$.\\
\end{corollary}

We will later prove a partial converse to the above. If $n \gg 0$, then the homological dimension of $\Conf_n^{\sink}(G,V)$ is exactly $|E|$ (see Theorem \ref{homdim}).\\

\begin{corollary}\label{eulerchar}
The Euler characteristic of $\Conf_n^{\sink}(G,V)$ is given by
\[
\chi^{(G,V,E)}_n := \chi(\Conf_n^{\sink}(G,V)) = \sum_{i=0}^{|E|}(-1)^i\binom{n}{i}\binom{|E|}{i}i!|V|^{n-i}.
\]
In particular, the Euler characteristic of $\Conf_n^{\sink}(G,V)$ depends only on $n$, $|E|$ and $|V|$.\\
\end{corollary}

\begin{proof}
We may compute the Euler characteristic as the alternating sum of the number of $i$-cells of $\DConf_n^{\sink}(G,V)$. To construct an $i$-cell, we must first choose which edges will appear, as well as which coordinates will be assigned to these edges. We then make an assignment of our chosen points to our chosen edges, and distribute the remaining $n-i$ coordinates to vertices. This gives us a total of $\binom{n}{i}\binom{|E|}{i}i!|V|^{n-i}$ $i$-cells, as desired.\\
\end{proof}

In view of Theorem \ref{sinkexpstab}, Corollary \ref{eulerchar} becomes very natural as it essentially states that the function
\[
n \mapsto \chi^{(G,V,E)}_n
\]
agrees with a function of the form $p(n)|V|^n$, where $p(n)$ is a polynomial of degree $|E|$.\\

Note that in Section \ref{treecaseI} we will use this cellular model to compute $\pi_1(\Conf_n^{\sink}(G,V))$ in the case wherein $(G,V,E)$ is a tree.\\

To conclude this section, we will use the first cellular model to prove that the spaces $\Conf_n^{\sink}(G,V)$ are always $K(\pi,1)$. We will accomplish this by proving that the universal cover of $\DConf_n^{\sink}(G,V)$ is always a CAT$(0)$ cubical complex. This strategy was implemented by Abrams in \cite[Theorem 3.10]{A} to prove that the classical configuration spaces of graphs are $K(\pi,1)$. See that source, and the references there in, for a brief introduction to the necessary background on CAT$(0)$ cubical complexes and the link condition.\\

\begin{theorem}\label{kpi1}
Let $(G,V,E)$ be a graph. Then the space $\Conf_n^{\sink}(G,V)$ is a $K(\pi,1)$ for all $n \geq 0$.\\
\end{theorem}

\begin{proof}
We begin by noting that while $\DConf_n^{\sink}(G,V)$ might not be a cubical complex (i.e. in the cases where $G$ contains loops) its universal cover will be. To prove that the universal cover is CAT$(0)$, it therefore suffices to show that $\DConf_n^{\sink}(G,V)$ satisfies the link condition (see \cite[Chapter 3]{A}).\\

Let $\sigma = \sigma_1 \times \ldots \times \sigma_n$ be an $i$-cell of $\DConf_n^{\sink}(G,V)$. A 0-cell in the link of $\sigma$ corresponds to an $(i+1)$-cell of $\DConf_n^{\sink}(G,V)$ which contains $\sigma$ as a face. In other words, it is an $(i+1)$-cell which is obtained from $\sigma$ by replacing some vertex, say $\sigma_j$, with an edge containing $\sigma_j$. If two 0-cells in the link of $\sigma$ are connected via a 1-cell, we note that the edge each 0-cell adds to $\sigma$ must be distinct. Given three 0-cells in the link of $\sigma$ which form a triangle in this link, we have that each 0-cell corresponds to replacing some vertex appearing in $\sigma$ with an edge containing that vertex. Each of these three edges must be distinct by the previous remark, and so we may form an $(i+3)$-cell by adding all three of these edges to $\sigma$. This cell will be our desired 2-simplex bounded by the triangle we started with.\\
\end{proof}

We note that the above theorem is not entirely surprising. In fact, it is actually somewhat common for configuration spaces to have this property. For instance, we have already discussed that the classical configuration spaces of graphs are $K(\pi,1)$, unless the graph is a circle or line segment \cite[Theorem 3.10]{A}. It is also a classically known fact that configuration spaces of genus at least 1 surfaces (i.e. closed manifolds of dimension 2 which are neither the sphere nor the real projective plane) are $K(\pi,1)$ (see, for example, \cite[Theorem 2.7]{CP}).\\

\subsection{The second model: applying discrete Morse Theory}

In this section we provide a second cellular model for $\Conf^{\sink}_n(G,V)$. This model will be obtained from the first via a discrete Morse theory refinement. We will use this model in the next section to fully compute the homology of $\DConf^{\sink}_n(G,V)$ in the cases where $(G,V,E)$ is a tree. To begin, we will need to define our discrete vector field. The vector field constructed in this section is heavily inspired by the vector field constructed by Farley and Sabakla for the usual configuration spaces of graphs \cite{FS}.\\

Fix a spanning tree $(T,V,E_T)$ of $G$, as well as an embedding of $T$ into the plane. Next, label the vertices of $T$ using a depth-first algorithm. That is, we begin with some vertex of $T$ of degree 1, and label it with the number 0. This vertex will henceforth be referred to as the \textbf{root} of $G$. Moving along this edge we label each vertex with subsequently bigger numbers until a vertex of degree at least 3 is reached. At this point, one chooses the path which is left most (with respect to our embedding into the plane) and continues. When a vertex of degree 1 is reached, one simply goes back to the most recently visited vertex of degree at least 3 and chooses the next left-most direction. If all directions lead to vertices which have already been labeled, then one goes back to the next most recently visited vertex of degree at most 3. An example of a correctly labeled tree is given in Figure \ref{labeledtree}.\\

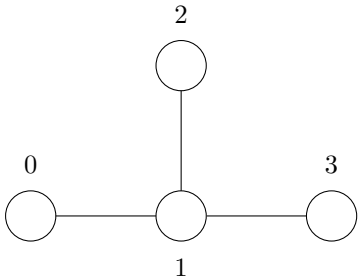
\begin{figure}
\begin{tikzpicture}

\tikzstyle{every node}=[draw,circle,fill=white,minimum size=19pt,
                            inner sep=0pt]
    \draw (0,0) node (0) [label=above:$0$]  {}
					--(2,0) node (2) [label=below:$1$] {}
					--(2,2) node (4) [label=above:$2$]{}
    			(2) -- (4,0) node (7) [label=above:$3$]{};

\end{tikzpicture}
\caption{A tree which is properly labeled.}\label{labeledtree}
\end{figure}

\begin{definition}\label{sinkdgvf}
Let $(G,V,E)$ and $(T,V,E_T)$ be as in the previous paragraph. Given an edge $e \in E$, we use $\tau(e) \in V$ to denote the endpoint of $e$ with the largest label. We will sometimes refer to $\tau(e)$ as the \textbf{top} of $e$. Similarly, we use $\iota(e)$ to denote the endpoint with the smallest label, or the \textbf{bottom} of $e$.\\

Let $\sigma = \sigma_1 \times \ldots \times \sigma_n$ be an $i$-cell of $\DConf^{\sink}_n(G,V)$ for some $i$ and $n$. If $\sigma_j \in V$ is not the root of $G$, then we write $e(\sigma_j)$ to denote the unique edge of $T$ for which $\sigma_j$ is the endpoint with the highest label.  We say that $\sigma_j$ is \textbf{blocked in $\sigma$} if $\sigma_j$ is the root of $G$, or $e(\sigma_j) = \sigma_l$ for some $l$.\\

We will define a collection of partially defined functions $X$ on the cells of $\DConf^{\sink}_n(G,V)$ inductively as follows. If $\sigma$ is in the image of $X_{i-1}$, then $X_i(\sigma)$ is undefined. If all vertices of $\sigma$ are blocked, then $X_i(\sigma)$ is undefined. Otherwise, let $\sigma_{i_1} = \ldots =\sigma_{i_r}$ be the vertex of $G$ in $\sigma$ which is unblocked with the lowest label, and assume that $i_1 < \ldots < i_r$. We set,
\[
X_i(\sigma) = \sigma_1 \times \ldots \times \sigma_{i_1-1} \times e(\sigma_{i_1}) \times \sigma_{i_1+1} \times \ldots \times \sigma_n
\]
\text{}\\
\end{definition}

Our goal for the first part of this section will be to show that $X$ is actually a discrete gradient vector field. Before we do this, it is important for one to develop a good intuition for what the function $X_i$ looks like for each $i$. An illustration is given in Figure \ref{vfex}.\\

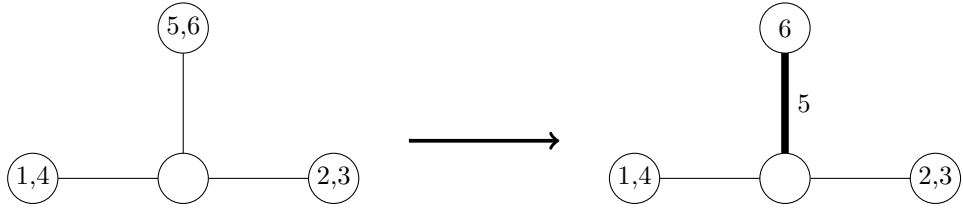
\begin{figure}
\begin{tikzpicture}

\draw [line width=1mm] (10,0) -- node[right] {5} (10,2);

\tikzstyle{every node}=[draw,circle,fill=white,minimum size=19pt,
                            inner sep=0pt]
    \draw (0,0) node (0)  {1,4}
					--(2,0) node (2) {}
					--(2,2) node (4) {5,6}
    			(2) -- (4,0) node (7) {2,3};
					
		\draw[->, line width=.5mm]
		     (5, .5) -- (7,.5);
			
		\draw (8,0) node (1)  {1,4}
					--(10,0) node (3) {}
					--(10,2) node (5) {6}
    			(3) -- (12,0) node (6) {2,3};

\end{tikzpicture}
\caption{An example of the function $X_0$ acting on a 0-cell of $\DConf_6^{\sink}(G,V)$, where the labeling is given in Figure \ref{labeledtree}. Note that the vertex labeled 2 was chosen, as it is the smallest unblocked vertex in the cell. Moreover, the coordinate 5 was chosen as it is the smallest coordinate index occupying the vertex.}\label{vfex}
\end{figure}

For the purpose of the next proof, we will need some convenient nomenclature.\\

\begin{definition}
Let $\sigma = \sigma_1 \times \ldots \times \sigma_n$ be a cell of $\DConf_n^{\sink}(G,V)$. If $e = \sigma_j$ is an edge of $T$ appearing in $\sigma$, we say that $e$ is \textbf{order respecting} if $j$ is strictly smaller than the index of every coordinate occupying $\tau(e)$. An edge $e \in E$ is said to be a \textbf{deleted edge} if $e \in E - E_T$. We note that, by definition, deleted edges are never order respecting.\\
\end{definition}

\begin{remark}
The terminology of order respecting was first used in the context of the usual configuration spaces of graphs by Farley and Sabalka \cite{FS}. We will find that order respecting edges in our context play a very similar role to order respecting edges in that context.\\
\end{remark}

\begin{proposition}\label{vfproof}
The collection of partially defined functions $X$ is a discrete vector field on $\DConf_n^{\sink}(G,V)$.\\
\end{proposition}

\begin{proof}
We must verify three things. These are:
\begin{enumerate}
\item that the image of $X_i$ is disjoint from the domain of $X_{i+1}$;
\item that for any $i$-cell $\sigma$ in the domain of $X_i$, $\sigma$ is a regular face of $X_i(\sigma)$;
\item that $X_i$ is injective.\\
\end{enumerate}
The first claim was built into our definition of the functions $X_i$. For the second claim, we note that any loops of $G$ are always excluded from our choice of spanning tree. Therefore, $X_i(\sigma)$ is obtained from $\sigma$ by replacing a vertex appearing in $\sigma$ by an interval. This will make $\sigma$ a regular face of $X_i(\sigma)$.\\

It therefore remains to show that $X_i$ is injective. Let $\sigma$ be a cell in the domain of $X_i$, and consider $X_i(\sigma)$. The cell $X_i(\sigma)$ must have at least one order respecting edge $e$. We note that if $\sigma'$ is an $i$-cell for which $X(\sigma') = X(\sigma)$, then there must be some order respecting edge $e$ of $X(\sigma)$ for which $\sigma'$ is obtained from $X(\sigma)$ by replacing $e$ with $\tau(e)$. Let $e$ be the order respecting edge which is added to $\sigma'$ to form $X(\sigma)$, and let $e_\sigma$ denote the order respecting edge of $X_i(\sigma)$ for which $\tau(e_\sigma)$ has the smallest label among all order respecting edges. We will show that $e = e_\sigma$.\\

Assume for contradiction that this is not the case. Let $\sigma''$ be the $i-1$ cell obtained from $X(\sigma)$ by replacing $e$ with $\tau(e)$ and $e_\sigma$ with $\tau(e_\sigma)$. The vertex $\tau(e_\sigma)$ is unblocked in $\sigma''$, and is the smallest unblocked vertex in this cell. Indeed, if there were some smaller unblocked vertex, then this vertex would also be unblocked in $\sigma'$ however, by assumption, $\tau(e)$ is the smallest unblocked vertex in $\sigma'$. Next we claim that $\sigma''$ is not in the image of $X_{i-2}$. Indeed, if it were then there would be some order respecting edge $e''$ in $\sigma''$ which was added by $X$ to create $\sigma''$. However, $e''$ is also order respecting in $X(\sigma)$, and therefore $\tau(e'')$ has smaller label than $\tau(e_\sigma)$. Therefore applying $X_{i-2}$ to the cell obtained from $\sigma''$ by replacing $e''$ with $\tau(e'')$ would not produce $\sigma''$, a contradiction.\\

In summary, we have shown that $\sigma''$ is not in the image of $X_{i-2}$, and that $\tau(e_\sigma)$ is the smallest unblocked vertex. Therefore $X_{i-1}(\sigma'')$ is defined and equal to $\sigma'$ by definition. This is a contradiction on the fact that $X(\sigma') = X(\sigma)$.\\
\end{proof}

To prove that $X$ is a discrete gradient vector field, it therefore remains to show that it does not admit any non-trivial closed paths (see Definition \ref{dmtmaindefs}). To accomplish this, we use a trick of Farley and Sabalka \cite[Theorem 3.8]{FS}.\\

\begin{definition}
Let $v$ be a vertex of $G$, and let $\sigma = \sigma_1 \times \ldots \times \sigma_n$ be a cell of $\DConf_n^{\sink}(G,V)$. The \textbf{geodesic} to $v$ is the unique path in $T$ between the root (i.e. the vertex of $G$ with label 0) and $v$. We define an integrally valued function on the cells of $\DConf_n^{\sink}(G,V)$,
\[
f_v(\sigma) = |\{j \mid \sigma_j \text{ is contained in the geodesic to $v$.}\}|
\]
\text{}\\
\end{definition}

Note that if multiple coordinates occupy the same vertex, then the function $f_v$ does count all of them. For example, if $(G,V,E)$ is the graph of Figure \ref{labeledtree}, $\sigma$ is the cell of Figure \ref{cellexample}, and $v$ is the vertex labeled 2, then $f_v(\sigma) = 4$. If we instead choose $v$ to be the vertex labeled 3, then $f_v(\sigma)$ is $4$, as well.\\

The following lemma appears in the work of Farley in Sabalka \cite{FS}, although they only deal with the usual configuration spaces of graphs. The proof is almost tautological based on the definition of $f_v$.\\

\begin{lemma}\label{fvcase}
For any vertex $v$ of $G$, $f_v$ satisfies the following properties:
\begin{enumerate}
\item if a cell $\sigma'$ can be obtained from a cell $\sigma$ by replacing some non-deleted edge $e$ with $\tau(e)$, then $f_v(\sigma) = f_v(\sigma')$. In particular, if $\sigma$ is a redundant cell, then $f_v(\sigma) = f_v(X(\sigma))$;
\item if a cell $\sigma'$ can be obtained from a cell $\sigma$ by replacing some non-deleted edge $e$ with $\iota(e)$, then either $f_v(\sigma) = f_v(\sigma')$, or $f_v(\sigma') = f_v(\sigma) + 1$. The latter case happens precisely when the geodesic to $v$ passes through $\iota(e)$ but does not contain $e$;
\item if a cell $\sigma'$ can be obtained from a cell $\sigma$ by replacing some deleted edge $e$ with $\tau(e)$, then either $f_v(\sigma) = f_v(\sigma')$, or $f_v(\sigma') = f_v(\sigma) + 1$. The latter case happens precisely when the geodesic to $v$ passes through $\tau(e)$;
\item if a cell $\sigma'$ can be obtained from a cell $\sigma$ by replacing some deleted edge $e$ with $\iota(e)$, then either $f_v(\sigma) = f_v(\sigma')$, or $f_v(\sigma') = f_v(\sigma) + 1$. The latter case happens precisely when the geodesic to $v$ passes through $\iota(e)$;
\end{enumerate}
\end{lemma}

As was noted by Farley and Sabalka in their original work \cite{FS}, one of the main benefits of the functions $f_v$ is that they behave fairly well when applied to cellular paths. In particular, if $\alpha_0^{(i)}, \alpha_1^{(i)}, \ldots, \alpha_{l-1}^{(i)}, \alpha_l^{(i)}$ is a cellular path between $\sigma$ and $\sigma'$, then for any vertex $v$
\begin{eqnarray}
f_v(\sigma_0^{(i)}) \leq f_v(\sigma_1^{(i)}) \leq \ldots \leq f_v(\sigma_{l-1}^{(i)}) \leq f_v(\sigma_l^{(i)}). \label{pathadd}
\end{eqnarray}
This will play a key part in the following proof.\\

\begin{proposition}
The discrete vector field $X$ does not permit any non-trivial closed cellular paths.\\
\end{proposition}

\begin{proof}
Assume that $\sigma^0, \sigma^1, \ldots, \sigma^{l-1}, \sigma^l$ is a closed cellular path. The observation (\ref{pathadd}) implies that, in fact, $f_v(\sigma^j) = f_v(\sigma^k)$ for any $j,k$. Because $v$ is arbitrary, this immediately disqualifies cases 2 through 4 of Lemma \ref{fvcase}. In particular, for each $j$ it must be the case that $\sigma^{j+1}$ is obtained from $X(\sigma^j)$ by replacing some non-deleted edge $e$ with $\tau(e)$.\\

Assume first that for some $j$, $\sigma^j$ does not have any order respecting edges. We claim that $\sigma^{j+1}$ must have at least one order respecting edge. Indeed, we know that $X(\sigma^j)$ is obtained from $\sigma^j$ by replacing some unblocked vertex, appearing in the least available coordinate index, with an edge $e$. This edge will, by definition, be order respecting. By the previous paragraph we know that $\sigma^{j+1}$ is a face of $X(\sigma^j)$ obtained by replacing an edge $e'$ with $\tau(e')$. If $e$ is chosen, then it would be the case that $\sigma^{j+1} = \sigma^j$, which we may assume is not the case. Therefore our claim is proven.\\

We may therefore assume without loss that $\sigma^j$ has at least one order respecting edge, as otherwise we replace $\sigma^j$ with $\sigma^{j+1}$ in the argument which follows. Let $e_j$ be the order respecting edge in $\sigma^j$ for which $\tau(e_j)$ is minimal among all order respecting edges in $\sigma^j$. Then we claim that the label of $\tau(e_{j+1})$ is strictly smaller than that of $\tau(e_j)$. If $v$ is the unblocked vertex of $\sigma^j$ of minimal label, then $\tau(e_j)$ has a strictly larger label than $v$ (see the proof of Proposition \ref{vfproof}). On the other hand, $X(\sigma^j)$ is obtained from $\sigma^j$ by replacing the smallest coordinate occupying $v$ with $e(v)$. By definition $e(v)$ is order respecting, and by the first paragraph it is also order respecting in $\sigma^{j+1}$ because $e_{j+1}$ was chosen to be minimal we must have that the label of $\tau(e_{j+1})$ is at most the label of $v$. This is proves the claim. Note that we have arrived at a contradiction on our path being closed, as $j$ was arbitrary.\\
\end{proof}

This now immediately leads to the main theorem of this section. We note that the theorem simply follows from the work in this section, as well as Section \ref{dmt}.\\

\begin{theorem}\label{dconfisdm}
The discrete vector field $X$ is actually a discrete gradient vector field. In particular:
\begin{enumerate}
\item $\DConf_n^{\sink}(G,V)$ is homotopy equivalent to a CW complex with precisely $m_i$ $i$-cells for each $i$, where $m_i$ is the number of critical $i$-cells of $X$;
\item the homology of $\DConf_n^{\sink}(G,V)$ is isomorphic to the homology of the Morse complex associated to $X$ (see Definition \ref{morsediff}).\\
\end{enumerate}
\end{theorem}

With the above theorem in mind, it becomes important that one develop a good understanding of the critical cells of the vector field $X$. The following classification follows from the work of this section.\\

\begin{proposition}\label{critclass}
Let $\sigma^{(i)}$ be a cell of $\DConf_n^{\sink}(G,V)$, then $\sigma^{(i)}$ is a critical cell of $X$ if and only if
\begin{enumerate}
\item all vertices appearing in $\sigma^{(i)}$ are blocked, and
\item $\sigma^{(i)}$ has no order respecting edges.
\end{enumerate}
In particular, the number of critical $i$-cells of $X$ is at most $\binom{|E|}{i} \binom{n}{i}i!\delta_i^{n-i}$, where $\delta_i = \min\{i+1,|V|\}$.\\
\end{proposition}

\begin{proof}
It is clear that the two conditions are sufficient for a cell to be critical. Conversely, assume that $\sigma$ is critical. If $\sigma$ has order respecting edges, then we let $e$ be the order respecting edge for which $\tau(e)$ is minimal. If the smallest unblocked vertex of $\sigma$ has smaller label than $\tau(e)$, then the cell is redundant, whereas if it has larger label than that of $\tau(e)$ the cell is collapsible (see the proof of Proposition \ref{vfproof}). We may therefore assume all vertices of $\sigma$ are blocked if it has an order respecting edge. In this case, it is clear that $\sigma$ is collapsible, as it is $X(\sigma')$ where $\sigma'$ is obtained from $\sigma$ by replacing $e$ with $\tau(e)$.\\

It follows that no edges of $\sigma$ are order respecting. In this case, having a unblocked vertex would clearly make $\sigma$ redundant.\\

For the second part of the claim, we can count cells whose every vertex is blocked as follows. Begin by choosing $i$ edges and assigning $i$ coordinates to these edges. At this point, every other coordinate must be placed at one of at most $\delta_i$ vertices (i.e the vertices which are $\tau(e)$ for one of the, at most $i$, chosen non-deleted edges $e$, as well as the root). The constructed cell may not be critical, although all critical cells arise in this fashion.\\
\end{proof}

The work in this section grants us a convenient bound on the Betti numbers of the spaces $\DConf_n^{\sink}(G,V)$.\\

\begin{corollary}\label{expbound}
For each $i$, the function
\[
n \mapsto \dim_\Q(H_i(\DConf_n^{\sink}(G,V)))
\]
is bounded from above by a function of the form $p(n)\delta_i^n$, where $p(n)$ is a polynomial of degree $i$ and,
\[
\delta_i = \min\{i+1,|V|\}
\]
\text{}\\
\end{corollary}

In later sections we will leverage Corollary \ref{expbound} to compute the homological dimension of $\DConf_n^{\sink}(G,V)$ for $n \gg 0$ (see Theorem \ref{homdim}).\\

\begin{remark}\label{oimod}
We note that the deformation retract of $\DConf_n(G,V)$ granted to us by discrete Morse theory will, in no obvious way, respect the action of $\FI_V$. One way to see this is to note that critical cells are not preserved by permuting coordinates due to the condition that no edge be order respecting. It is perhaps interesting to note that critical cells do have a natural action by the category $\OI_V$, which is the subcategory of $\FI_V$ whose morphisms are pairs $(f,g)$ where $f$ is an monotone (increasing) injection. Unfortunately, it is unclear whether this action commutes with the Morse differential, unless $(G,V,E)$ is a tree (see Theorem \ref{trivdiff}).\\
\end{remark}

In the next section we will compute the Morse differential (see Definition \ref{morsediff}) in the case where $(G,V,E)$ is a tree. This computation was accomplished by Farley for unordered configurations of trees in \cite{Far}. The techniques used in Farley's proof will not apply to this context, although certain core ideas of that work will be used. Note that there is very little work towards computing the Morse differential for arbitrary graphs in the classical configuration space setting, and this problem seems quite difficult (see \cite{KKP}\cite{KP}\cite{BF}).\\

\subsection{Applications of the two models to the case of trees}\label{treecaseI}

In this section we assume that $(G,V,E)$ is a tree. We will find that this class of graphs have sink configuration spaces which are much easier to study those of general graphs. One of the main tricks of this section will be to induct on the number of edges of $G$. To begin, we compute $\pi_1(\Conf_n^{\sink}(G,V))$.\\

\begin{theorem}\label{treepi}
Let $(G,V,E)$ be a tree. Then $\pi_1(\Conf_n^{\sink}(G,V))$ admits a presentation with $|E|\cdot((n-2)2^{n-1}+1)$ many generators wherein all relations are commutators.\\
\end{theorem}

\begin{proof}

We will proceed by using Van Kampen's Theorem and induction on the number of edges of $G$. Begin by choosing a vertex of $G$ of valency 1, $v_0$. Call the unique edge adjacent to $v_0$, $e_0$, and call the other endpoint of $e_0$, $v_1$. Finally, let $(G',V',E')$ be the subtree of $(G,V,E)$ obtained by removing $e_0$ and $v_0$ (equivalently, by contracting $\overline{e_0}$).\\

To use Van Kampen's theorem, we will need to cover $\Conf_n^{\sink}(G,V)$ by sufficiently nice open sets. We will work instead with the first cellular model, $\DConf_n^{\sink}(G,V)$. Let $[n]^{\{\overline{e_0},G'\}}$ denote the collection of set maps $\phi:[n] \rightarrow \{\overline{e_0},G'\}$. Then we may define $A_{\phi}$ to be the subcomplex of $\DConf_n^{\sink}(G,V)$ of all tuples $(x_1,\ldots,x_n)$ such that $x_i$ is an element of $\phi(i)$. For example, if $(G,V,E)$ is the tree of Figure \ref{labeledtree}, and $v_0$ is the vertex labeled 3, then the cell of Figure \ref{cellexample} is an example of a cell in $A_{\phi}$ where $\phi$ is the map
\[
\phi(\{1,4,5,6\}) = G', \phi(\{2,3\}) = \overline{e_0}.
\]
For each $\phi \in [n]^{\{\overline{e_0},G'\}}$ we may find some open subset $U_\phi$ of $\DConf_n^{\sink}(G,V)$ which deformation retracts onto $A_\phi$. In particular, we may set $U_\phi$ to be $A_\phi$ with the added condition that coordinates located on $v_1$ may move within a small ball. Then $\{U_\phi\}_{\phi}$ is an open cover of $\DConf_n^{\sink}(G,V)$. Note that because we are working within $\DConf_n(G,V)$, it is never the case that more than one coordinate is on the interior of an edge of $G$ at any time. Therefore, $U_\phi$ may be deformation retracted onto $A_\phi$ by simply sliding any points which strayed from $v_1$ back onto $v_1$. For the same reason, any intersection $\cap_{\phi} U_\phi$ will deformation retract onto the intersection $\cap_\phi A_\phi$. We may therefore proceed with our Van Kampen computation by computing $\pi_1(A_\phi)$. First observe that
\begin{eqnarray}
A_\phi &\cong& \DConf^{\sink}_{|\phi^{-1}(G')|}(G',V') \times \DConf^{\sink}_{|\phi^{-1}(\overline{e_0})|}(\overline{e_0},\{v_0,v_1\}),\\ \label{observation1}
A_\phi \cap A_\psi &\cong& \DConf^{\sink}_{|\psi^{-1}(G') \cap \phi^{-1}(G')|}(G',V') \times \DConf^{\sink}_{|\psi^{-1}(\overline{e_0}) \cap \phi^{-1}(\overline{e_0})|}(\overline{e_0},\{v_0,v_1\}).\label{observation2}
\end{eqnarray}
Also observe that $\cap_{\phi} A_\phi$ is precisely $\{(v_1,\ldots,v_1)\}$. Moreover, it is easily seen that for any triple of maps $\phi_1,\phi_2,\phi_3$, that $\cap_{i} A_{\phi_i}$ is connected. It follows that we may use the cover $\{U_\phi\}_\phi$ to compute $\pi_1(\DConf_n^{\sink}(G,V),(v_1,\ldots,v_1))$. Because $\DConf_n^{\sink}(G,V)$ is path connected, this will suffice. We have,
\begin{eqnarray}
\pi_1(\DConf_n^{\sink}(G,V)) = \star_{\phi} \pi_1(A_\phi) / N, \label{vct}
\end{eqnarray}
where $\star$ is the free product operator of groups, and $N$ is the normal subgroup of $\star_{\phi} \pi_1(A_\phi)$ generated by the inclusions of $\pi_1(A_\phi \cap A_\psi)$ into $\pi_1(A_\phi)$ and $\pi_1(A_\psi)$ for each pair of maps $\phi,\psi$.\\

The observations (\ref{observation1}) and (\ref{observation2}) allow us to compute
\begin{eqnarray*}
\pi_1(A_\phi \cap A_\psi) \cong  &\pi_1(\DConf_{|\psi^{-1}(G') \cap \phi^{-1}(G')|}(G',V'))&\\ 
                                           &\bigoplus \pi_1(\DConf_{|\psi^{-1}(\overline{e_0}) \cap \phi^{-1}(\overline{e_0})|}(\overline{e_0},\{v_0,v_1\})).&
\end{eqnarray*}
We visualize loops here as a loop using $|\psi^{-1}(G') \cap \phi^{-1}(G')|$ coordinates moving around $G'$, $|\psi^{-1}(\overline{e_0}) \cap \phi^{-1}(\overline{e_0})|$ coordinates moving around $e_0$, and all remaining coordinates sitting stationary on $v_1$. Such loops will exist within both $A_\phi$ and $A_\psi$, and Van Kampen's theorem amounts to saying that such loops should be identified in the free product $\star_{\phi} \pi_1(A_\phi)$. Let $\phi_0 \in [n]^{\{\overline{e_0},G'\}}$ be the map which sends all elements of $[n]$ to $G'$, and let $\psi_0$ be the map which sends every element to $e_0$. Then for any $\phi$, the fundamental group of $A_{\phi_0} \cap A_{\phi}$ is generated by precisely those loops in $A_{\phi}$ which don't use the edge $e_0$. Similarly, $A_\phi \cap A_{\psi_0}$ are those loops of $A_\phi$ which don't use $G'$. The description (\ref{observation1}) implies that these two classes of loops actually generate $\pi_1(A_\phi)$. By varying across all choices of $\phi$, it follows that all terms in the free product (\ref{vct}) are killed in the quotient, except for $\pi_1(A_{\phi_0})$ and $\pi_1(A_{\psi_0})$. Moreover, we have to account for the relations that loops which only use $G'$, and loops which only use $\overline{e_0}$, commute with one another in $\pi_1(A_\phi)$. For each loop $\gamma^{(m)} \in \pi_1(\Conf^{\sink}_{m}(G',V'))$, let $\gamma^{(m)}_{\phi}$ be the loop in $\pi_1(\Conf_n^{\sink}(G,V))$ which replicates $\gamma^{(m)}$ using the coordinates in $\phi^{-1}(G')$, while keeping the coordinates of $\phi^{-1}(\overline{e_0})$ fixed on $v_1$. Similarly define $\alpha^{(m)}_\phi$ for any $\alpha^{(m)} \in  \pi_1(\Conf^{\sink}_{m}(\overline{e_0},\{v_0,v_1\})$. Then,
\[
\pi_1(\Conf_n^{\sink}(G,V)) \cong \pi_1(\Conf^{\sink}_n(G',V')) \star \pi_1(\Conf^{\sink}_{n}(\overline{e_0},\{v_0,v_1\})) / N'
\]
where $N'$ is the normal subgroup generated by commutators $[\gamma^{(m)}_\phi,\alpha^{(m-|\phi^{-1}(G')|)}_\phi]$. We note that $\DConf_n^{\sink}(\overline{e_0},\{v_0,v_1\})$ is a 1-dimensional CW-complex, and so its fundamental group is free. The rank of $\pi_1(\DConf_n^{\sink}(\overline{e_0},\{v_0,v_1\})$ was computed in \cite[Proposition 2.2]{CL} to be $(n-2)2^{n-1} + 1$.  Our result now follows by induction.\\
\end{proof}

In the specific case wherein $(G,V,E)$ is a line segment with three vertices, the above proof implies that $\pi_1(\Conf_n^{\sink}(G,V,E))$ admits a presentation for which every relation is a commutator of generators. Such groups are usually called \textbf{right-angled Artin}. See \cite{Cha} for an introduction to the theory of these groups.\\

Our next goal will be to compute the homology groups of $\Conf_n^{\sink}(G,V)$ using the second model.\\

Assume that the vertices of $G$ are labeled as in Section \ref{dmtgraph}. We let $(G',V',E')$ be the subtree of $(G,V,E)$ which does not contain the vertex of largest label, i.e. the tree obtained $(G,V,E)$ by contracting the unique edge adjacent to the vertex of highest label. This induces an ordering on the vertices of $G'$. Call the vertex of highest label of $G$, $v_0$ and its incident edge $e_0$. Our first key observation will allow us to relate properties of the space $\Conf_n^{\sink}(G,V)$ to those of $\Conf_n^{\sink}(G',V')$.\\

\begin{definition}
Let $(G,V,E)$ and $(G',V',E')$ be as above. Then there is a retraction $r_n:\Conf_n^{\sink}(G,V) \rightarrow \Conf_n^{\sink}(G',V')$ induced by contraction of $e_0$. In particular, the inclusion $\Conf_n^{\sink}(G',V') \hookrightarrow \Conf_n^{\sink}(G,V)$ induces a split injection for each $i$, $H_i(\Conf_n^{\sink}(G',V')) \hookrightarrow H_i(\Conf_n^{\sink}(G,V))$. \\
\end{definition}

Note that the existence of the retraction is almost entirely unique to trees. While one may always define a map induced by contracting an edge, there is no general way to embed the resulting graph back into the original graph. This elementary property of trees will allow us to provide induction arguments for many theorems which strengthen the work of the previous section.\\

Our goal will be to prove that the Morse differential is trivial in the case of trees. To prove this, we will need to develop a better understanding of the Morse differential in this case. While we have previously defined this differential, the spaces $\DConf_n(G,V)$ are nice enough that we may simplify the previously described expression. Note that the following description appears in \cite{Far}\\

\begin{definition}\label{morsediffrefine}
As before, let $\K^G_{i,n}$ be the free $\Z$-module on the $i$-cells of $\DConf_n^{\sink}(G,V)$ and let $\M_{i,n}^G$ be the free $\Z$ module on the critical cells of $G$. We identify each edge $e$ of $E$ with the interval $[0,1]$ by setting $\tau(e) = 0, \iota(e) = 1$. In this case, the differential $\partial:\K^G_{i,n} \rightarrow \K^G_{i-1,n}$ can be expressed
\[
\partial(\sigma_1 \times \ldots \times \sigma_n) = \sum_{\sigma_j \text{ an edge of $\sigma$}} \left( (\sigma_1 \times \ldots \times \sigma_{j-1} \times \iota(\sigma_j) \times \sigma_{j+1}) - (\sigma_1 \times \ldots \times \sigma_{j-1} \times \tau(\sigma_j) \times \sigma_{j+1}) \right).
\]
Next, we observe that the discrete vector field $X$ induces a map $\K_{i,n}^G  \rightarrow \K_{i+1,n}^G$ by
\[
X(\sigma) = \begin{cases} X(\sigma) &\text{ if $\sigma$ is redundant}\\ 0 &\text{ otherwise}\end{cases}
\]
Then we may define $F:\K_{i,n}^G \rightarrow \K_{i,n}^G$, 
\[
F(\sigma) = (1+\partial X)(\sigma)
\]
Because $X$ lacks non-trivial closed paths, it is a fact that $F^m(\sigma) = F^{m+1}(\sigma)$ for all $m \gg 0$ \cite{Fo1}, and we set $F^\infty$ to be this stable value. Then the Morse differential is given by
\[
\widetilde{\partial} := \pi F^{\infty} \partial
\]
where $\pi:\K_{i,n}^G \rightarrow \M_{i,n}^G$ is the projection onto the critical cells.\\
\end{definition}

To make more sense of the proof of Theorem \ref{trivdiff}, we take a moment to develop a visual picture of what it's action looks like. One begins with a critical cell $\sigma$, and produces all faces of $\sigma$ by \textbf{disassembling} its edges one-by-one (see Figure \ref{disassemble}). These faces are then put in a linear combination, accounting for signs appropriately. Each of these faces $\sigma'$ will either be redundant, or collapsible. In the latter case, the cell is removed from the linear combination. In the redundant case, One produces all faces of $X(\sigma')$, once again by disassembling its edges one by one. These faces are then placed in linear combination. The cell $\sigma'$ is replaced in our linear combination by the new linear combination of faces of $X(\sigma')$, excluding $\sigma'$ itself. This processes is repeated, until all cells which appear are critical. In particular, we think of the Morse differential as beginning with the faces of the inputted critical cells and ``flowing,'' in all possible ways using the vector field $X$, from these faces to critical cells.

\begin{figure}
\begin{tikzpicture}

\draw [line width=1mm] (1.5,0) -- node[above] {3} (3,0);
\draw [line width=1mm] (1.5,0) -- node[right] {5} (1.5,1.5);
\draw [line width=1mm] (7.5,0) -- node[above] {3} (9,0);
\draw [line width=1mm] (14.5,0) -- node[above] {3} (16,0);
\draw [line width = .5mm, decoration={brace,raise=5pt},decorate]
			(5.5,-0.5)--(5.5,2);
\draw [line width = .5mm, decoration={brace,mirror,raise=5pt},decorate]
			(16.5,-0.5)--(16.5,2);

\tikzstyle{every node}=[draw,circle,fill=white,minimum size=19pt,
                            inner sep=0pt]
    \draw (0,0) node (0)  {1,4}
					--(1.5,0) node (2) {}
					--(1.5,1.5) node (4) {6}
    			(2) -- (3,0) node (7) {2};
					
		\draw[->, line width=.5mm]
		     (3.5, .5) -- (4.5,.5);
		\draw (6,0) node (1)  {1,4}
					--(7.5,0) node (3) {5}
					--(7.5,1.5) node (5) {6}
    			(3) -- (9,0) node (6) {2};
					
		\draw[-, line width=.5mm]
		     (10.5, .5) -- (11.5,.5);
			
		\draw (13,0) node (8)  {1,4}
					--(14.5,0) node (9) {}
					--(14.5,1.5) node (10) {5,6}
    			(9) -- (16,0) node (11) {2};

\end{tikzpicture}
\caption{An example of producing two faces of a cell by disassembling an edge. Here, the edge being disassembled is the edge occupied by coordinate index 5. If we label our tree's vertices as in Figure \ref{labeledtree}, then the signs of the two faces are determined by whether or not the 5 is placed on a smaller vertex (positive) or a larger one (negative).}\label{disassemble} 
\end{figure}
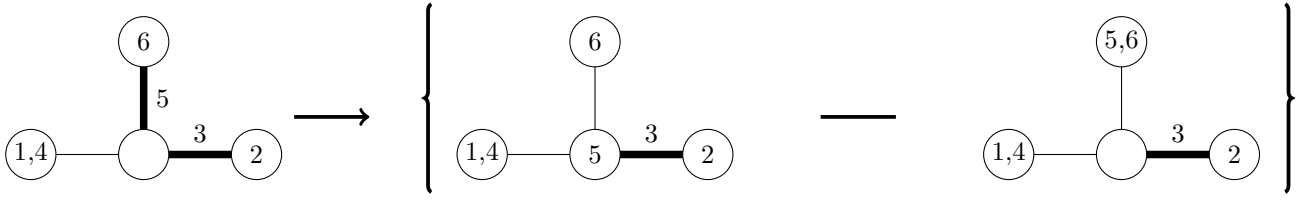

\begin{lemma}
The inclusion $\DConf_n^{\sink}(G',V') \hookrightarrow \DConf_n^{\sink}(G,V)$ respects the discrete Morse structure on these spaces. That is to say, cells of $\DConf_n^{\sink}(G',V')$ are critical, redundant, or collapsible if and only if they are critical, redundant, or collapsible, respectively, in $\DConf_n^{\sink}(G,V)$. Moreover, the inclusion induces an inclusion of complexes $\M_{\dt,n}^{G'} \hookrightarrow \M_{\dt,n}^{G}$.\\
\end{lemma}

\begin{proof}
This follows from the fact that we chose $v_0$ to have the highest label in $G$. If $\sigma$ is a critical cell of $\DConf_n^{\sink}(G',V')$, then all vertices of $\sigma$ must be blocked, and all edges must be order respecting. Applying the inclusion into $\DConf_n^{\sink}(G,V)$ cannot create unblocked vertices, as $v_0$ has the highest label in $G$, and it clearly cannot created order respecting edges. Similar arguments work for the other two types of cells as well.\\\
\end{proof}

\begin{lemma}\label{lem1}
There is an exact sequence of complexes,
\[
0 \rightarrow \M_{\dt,n}^{G'} \rightarrow \M_{\dt,n}^{G} \rightarrow Q_{\dt,n}  \rightarrow 0,
\]
inducing isomorphisms for all $i \geq 0$,
\[
H_i(\Conf_n^{\sink}(G,V)) \cong H_i(\Conf_n^{\sink}(G',V')) \oplus H_i(Q_{\dt,n} ).
\]
\text{}\\
\end{lemma}

\begin{proof}
From the previous lemma we know that the inclusion $\Conf_n^{\sink}(G',V') \hookrightarrow \Conf_n^{\sink}(G,V)$ induces an exact sequence of complexes 
\[
0 \rightarrow \M_{\dt,n}^{G'} \rightarrow \M_{\dt,n}^{G} \rightarrow Q_{\dt,n}  \rightarrow 0.
\]
The long exact sequence on homology yields
\[
\ldots \rightarrow H_i(\M_{\dt,n}^{G'}) \rightarrow H_i(\M_{\dt,n}^{G}) \rightarrow H_i(Q_{\dt,n} ) \rightarrow \ldots,
\]
We know that $H_i(\M_{\dt,n}^{G'}) \cong H_i(\Conf_n^{\sink}(G',V'))$, and $H_i(\M_{\dt,n}^{G}) \cong H_i(\Conf_n^{\sink}(G,V))$. Moreover, by construction, the map $H_i(\M_{\dt,n}^{G'}) \rightarrow H_i(\M_{\dt,n}^{G})$ is induced by the inclusion of $\Conf_n^{\sink}(G',V')$ into $\Conf_n^{\sink}(G,V)$. This induced map is a split injection on homology by previous remarks. This concludes the proof.\\
\end{proof}

\begin{remark}
It will follows from the proof of the next theorem that the exact sequence
\[
0 \rightarrow \M_{\dt,n}^{G'} \rightarrow \M_{\dt,n}^{G} \rightarrow Q_{\dt,n} \rightarrow 0,
\]
splits. While one might think this is obvious, considering that the inclusion of $\Conf_n^{\sink}(G',V')$ into $\Conf_n^{\sink}(G,V)$ admits a retraction, it is actually a bit more challenging. Indeed, it is actually non-obvious whether the map induced by this retraction respects the Morse structure on the spaces. The issue here is that it is possible to apply the retraction to redundant cells, and obtain either a critical or collapsible cell in return.\\
\end{remark}

These two lemmas are all we will need to prove that the Morse differential is actually trivial.\\

\begin{theorem}\label{trivdiff}
Let $(G,V,E)$ be a tree. Then the Morse differential of $\M_{\dt,n}^{G}$ is trivial.\\
\end{theorem}

\begin{proof}
We will proceed by induction on the number of edges of $G$. If $G$ has zero edges, then the claim is clear. Otherwise, let $(G',V',E')$ be the subtree of $(G,V,E)$ described in the beginning of this section. Then Lemma \ref{lem1} implies that there is an exact sequence of complexes
\[
0 \rightarrow \M_{\dt,n}^{G'} \rightarrow \M_{\dt,n}^{G} \rightarrow Q_{\dt,n}  \rightarrow 0.
\]
By induction, the differential of $\M_{\dt,n}^{G'}$ is trivial, and we claim that the differential of $Q_{\dt,n} $ is trivial as well. Note that, for any $i$, $Q_{i,n}$ is the free $\Z$-module on critical cells which contain the edge $e_0$. If $\sigma \in Q_{i,n}$ is a critical cell, then those faces of $\sigma$ produced by disassembling $e_0$ must be sent to zero by $\pi F^{\infty}$. Calling either such face $\alpha$, any cell appearing in $F^{\infty}(\alpha)$ will either not contain $e_0$, or $e_0$ will be order respecting. This is because $v_0$ was chosen to be the largest vertex of $G$. In either case, the differential of $Q_{\dt,n} $ will map these cells to zero.\\

It follows that the only faces of $\sigma$ which are not immediately sent to zero are those which contain the edge $e_0$. For the same reasons as those given above, as one repeatedly applies $F$, anytime a cell is chosen by disassembling $e_0$ all cells which result in $F^{\infty}$ will be sent to zero by $\pi$. Therefore, the only cases where the differential of $Q_{\dt,n} $ might produce non-trivial terms is when one proceeds by a path that never disassembles $e_0$. However, such a path will necessarily be the same as a path which exists entirely within $\DConf_m^{\sink}(G',V')$, where $m$ is $n$ minus the number of coordinates which occupy either $e_0$ or $v_0$. By how the Morse differential is defined, this will be indistinguishable from an application of the differential of $\M_{\dt,m}^{G'}$ on the critical $(i-1)$-cell of $\DConf_m^{\sink}(G',V')$ obtained from $\sigma$ by ignoring the coordinates occupying $e_0$ and $v_0$, and reordering the coordinates which remain (See Figure \ref{ignore}). By induction this is zero. We conclude that the differential of both $\M_{\dt,n}^{G'}$ and $Q_{\dt,n} $ are trivial. Lemma \ref{lem1} implies that $H_i(\Conf_n^{\sink}(G,V))$ is free on the critical cells of $\Conf_n^{\sink}(G,V)$. However, the only way a complex of free $\Z$-modules can have homology which is term-wise isomorphic to the complex is if the complex has trivial differentials. This completes the proof.\\
\end{proof}

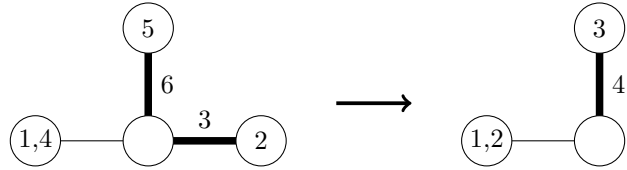
\begin{figure}
\begin{tikzpicture}
\draw [line width=1mm] (1.5,0) -- node[above] {3} (3,0);
\draw [line width=1mm] (1.5,0) -- node[right] {6} (1.5,1.5);
\draw [line width=1mm] (7.5,0) -- node[right] {4} (7.5,1.5);

\tikzstyle{every node}=[draw,circle,fill=white,minimum size=19pt,
                            inner sep=0pt]
    \draw (0,0) node (0)  {1,4}
					--(1.5,0) node (2) {}
					--(1.5,1.5) node (4) {5}
    			(2) -- (3,0) node (7) {2};
		\draw[->, line width=.5mm]
		     (4, .5) -- (5,.5);
		\draw (6,0) node (1)  {1,2}
					--(7.5,0) node (3) {}
					--(7.5,1.5) node (5) {3};
		 
\end{tikzpicture}
\caption{An example of taking a critical cell $\sigma \in Q_{\dt,6}$, and obtaining a critical cell of $\DConf_4^{\sink}(G',V')$ by ignoring $v_0$ and $e_0$, and relabeling the remaining coordinates. Note that our tree is labeled as in Figure \ref{labeledtree}.}\label{ignore}
\end{figure}

This theorem has a plethora of implications, which we record now. Note that there are some implications which we will hold off on addressing until Section \ref{repstab}.\\

\begin{corollary}\label{bigtreecor}
Let $(G,V,E)$ be a tree. Then for all $n,i \geq 0$
\begin{enumerate}
\item $H_i(\Conf_n^{\sink}(G,V))$ is torsion free;
\item the homological dimension of $\Conf_n^{\sink}(G,V)$ is given by
\[
\hdim(\Conf_n^{\sink}(G,V)) = \min\{\lfloor \frac{n}{2} \rfloor, |E|\}
\]
\item the groups $H_i(\Conf_n^{\sink}(G,V))$ depend only on $i,n,$ and $|E|$.\\
\end{enumerate}
\end{corollary}

\begin{proof}
The first statement is clear from Theorem \ref{trivdiff}.\\

For the second statement, we must show that $\DConf_n^{\sink}(G,V)$ has a critical $i$-cell if and only if $n \geq 2i$. Proposition \ref{critclass} tells us that a cell is critical if and only if all its vertices are unblocked, and none of its edges are order respecting. Any edge $e$ for which $\tau(e)$ is unoccupied is necessarily order respecting in the cell. Therefore, one must have at least $2i$ coordinates to be able to construct a critical $i$-cell, one coordinate for each edge and one coordinate for each $\tau(e)$. On the other hand, if $n = 2i$, then one may form a critical cell by choosing $i$-edges, assigning $\{i+1,\ldots,2i\}$ to these edges (in any order), and assigning $\{1,\ldots,i\}$ to the necessary vertices (once again, in any order).\\

For the final statement, each group is a free abelian group on the number of critical cells. To construct a critical cell, one first chooses $i$-edges and assigns each a coordinate. From this point, every other coordinate must occupy a vertex at the top of a chosen edge. Moreover, at least one coordinate occupying each vertex must have lower index than the coordinate occupying the associated edge. The number of ways to do this does not depend on the structure of the graph $G$.\\
\end{proof}

\section{$\Conf_n^{\sink}(G,V)$ and generalized representation stability}\label{repstab}

Once again, we fix a graph $(G,V,E)$ for the remainder of this section.\\

In this section we will primarily concerned with applying the techniques of asymptotic algebra of the space $\Conf_n^{\sink}(G,V)$. In particular, we will further explore the action of $\FI_V$ on these spaces (see Definition \ref{fivaction}), and the implications that produces. To begin, we show that, for any $i$, the $\FI_V$-module $H_i(\Conf_\dt^{\sink}(G,V))$ is finitely generated.\\

\begin{theorem}\label{fgmod}
For any $i$, the $\FI_V$-module $H_i(\Conf_\dt^{\sink}(G,V))$ is finitely generated.\\
\end{theorem}

\begin{proof}
In accordance with Remark \ref{dfiv}, it suffices to prove the claim for the $\FI_V$-module $H_i(\DConf_\dt^{\sink}(G,V))$. Let $\C_{n,i}$ denote the free $\Z$-module on the $i$-cells of $\DConf_n^{\sink}(G,V)$. Then we observe that the sum
\[
\C_{\dt,i} := \bigoplus_n \C_{n,i}
\]
inherits the structure of an $\FI_V$-module from the action on $\DConf_n^{\sink}(G,V)$. In particular, given a cell $\sigma^{(i)} = \sigma_1 \times \ldots \times \sigma_n$, and a morphism $(f,g):[n] \rightarrow [m]$ in $\FI_V$, we obtain
\[
(f,g)_\as(\sigma) = \sigma'_1 \times \ldots \times \sigma'_m
\]
where,
\[
\sigma'_j = \begin{cases} \sigma_{f^{-1}(j)} & \text{ if $j$ is in the image of $f$}\\ g(j) & \text{ otherwise.}\end{cases}
\]
We claim that the $\FI_V$-module $\C_{\dt,i}$ is finitely generated. In fact, we claim that it is generated in degree $i$.\\

Let $\sigma^{(i)}$ be an $i$-cell of $\DConf_n^{\sink}(G,V)$, where $n > i$. Then $\sigma$ must have some coordinate occupying a vertex. Say that $\sigma_j$ is this vertex. Let $(f,g):[n-1] \rightarrow [n]$ be the morphism defined by
\[
f(x) = \begin{cases} x &\text{ if $x < j$}\\ x+1 &\text{ if $x \geq j$.}\end{cases}, \hspace{1cm} g(j) = \sigma_j.
\]
Then we see that $\sigma$ is in the image of $(f,g)_\as$.\\

On the other hand, for all $n$, we have a complex
\[
\C_{n,\dt}: 0 \rightarrow \C_{n,|E|} \rightarrow \ldots \rightarrow \C_{n,0} \rightarrow 0
\]
whose homology is the homology of the space $\DConf_n^{\sink}(G,V)$. It is clear from definitions that the differentials in the above complex respect the action of $\FI_V$. That is to say, we have that 
\[
\C_{\dt',\dt}: 0 \rightarrow \C_{\dt',|E|} \rightarrow \ldots \rightarrow \C_{\dt',0} \rightarrow 0
\]
is a complex of $\FI_V$-modules. The Noetherian property (Theorem \ref{noeth}) implies our Theorem.\\
\end{proof}

\begin{remark}
We take a moment to acknowledge something which appeared in the above proof. Recall that, for any finite set $S$, we have a free module $M(S)$ (see Definition \ref{freemod}). It isn't hard to see that the complex $\C_{\dt',\dt}$ defined in the above theorem can actually be written
\[
0 \rightarrow M(E) \rightarrow \ldots \rightarrow \bigoplus_{\substack{E' \subseteq E\\ |E'| = i}} M(E') \rightarrow \ldots \rightarrow M(0) \rightarrow 0.
\]
\end{remark}

As an immediate consequence of Theorem \ref{fgmod}, as well as the work of Section \ref{fiddef}, we obtain the following.\\

\begin{corollary}\label{fivcor}
The modules $H_i(\Conf_n^{\sink}(G,V))$ enjoy the following properties:
\begin{enumerate}
\item there exist polynomials $p_1^{(i)},\ldots,p_{\delta_i}^{(i)} \in \Q[x]$ such that for all $n \gg 0$
\[
\rank_\Z(H_i(\Conf_n^{\sink}(G,V))) = p_1^{(i)}(n) + p_2^{(i)}(n)2^n + \ldots + p_{\delta_i}^{(i)}(n)\delta_i^n
\]
where
\[
\delta_i = \min\{i+1,|V|\};
\]
\item there exists a finite collection of elements $\{v_j\} \subseteq \bigoplus_n H_i(\Conf_n^{\sink}(G,V))$ such that every element of $H_i(\Conf_n^{\sink}(G,V))$ is $\Z$-linear combination of images of the $v_j$ under the action of $\FI_V$;
\item there is a finite integer $e_i^{(G,V,E)}$, independent of $n$, for which the exponent of $H_i(\Conf_n^{\sink}(G,V))$ is at most $e_i^{(G,V,E)}$.\\
\end{enumerate}
\end{corollary}

\begin{proof}
The first statement follows from Theorem \ref{noeth}, as well as Corollary \ref{expbound}.\\

The second statement is a rephrasing of the definition of finite generation as an $\FI_V$-module.\\

For the final statement, let $T^i$ be the $\FI_V$-submodule of $H_i(\Conf_\dt^{\sink}(G,V))$ for which $T^i_n$ is the collection of all $\Z$-torsion elements of $H_i(\Conf_n^{\sink}(G,V))$. By the Noetherian property we know that $T^i$ is finitely generated, because $H_i(\Conf_\dt^{\sink}(G,V))$ is. It follows that the exponent of $H_i(\Conf_n^{\sink}(G,V))$ will be bounded by the LCM of the orders of the (finitely many) generators of $H_i(\Conf_\dt^{\sink}(G,V))$.\\
\end{proof}

It is unclear at this time whether or not $H_i(\Conf_n^{\sink}(G,V))$ will have torsion if $(G,V,E)$ is not a tree. The above corollary tells us that any torsion would have uniformly bounded exponent, and would necessary appear before some finite $n$. Note that the homologies of the usual configuration spaces of graphs are always torsion free \cite[Theorem 1]{CL}, while the so-called unordered configuration spaces of graphs can have torsion in their homology \cite[Theorem 3.6]{KP}. We saw previously that $H_i(\Conf_n^{\sink}(G,V))$ is always torsion free if $(G,V,E)$ is a tree (Corollary \ref{bigtreecor})\\

Our next goal will be to develop a better understanding of the polynomials $p_j^i$ of Corollary \ref{fivcor}. As a consequence, we will be able to compute the homological dimensions of the spaces $\Conf_n^{\sink}(G,V)$ for $n$ sufficiently large.\\

\begin{theorem}\label{homdim}
Let $p_{|V|}^{(|E|)}(n)$ be as in Corollary \ref{fivcor}. Then $p_{|V|}^{(|E|)}(n)$ has degree $|E|$. In particular, for $n \gg 0$, the space $\Conf_n^{\sink}(G,V)$ has homological dimension $|E|$.\\
\end{theorem}

\begin{proof}
Corollary \ref{expbound} implies that, for all $i$, the Betti numbers $\rank_\Z(H_i(\Conf^{\sink}_n(G,V))$ are bounded by a function of the form $p(n)\delta_i^n$, where $p(n)$ is a polynomial of degree $i$. On the other hand, Corollary \ref{eulerchar} implies that the Euler characteristic of $\Conf_n^{\sink}(G,V)$, as a function of $n$, is of the form $q(n)|V|^n$ where $q(n)$ is a polynomial of degree $|E|$. Using the fact that the Euler characteristic can be realized as the alternating sum of Betti numbers, we note that the contributions of $H_i(\Conf_n^{\sink}(G,V))$ with $i < |E|$ will asymptotically grow strictly slower than $q(n)|V|^n$. Our theorem follows immediately from this.\\
\end{proof}

It isn't hard to see that finite generation of the $\FI_V$-module $H_i(\Conf_\dt^{\sink}(G,V))$ implies the same about the $\FI_V$-module $H_i(\Conf_\dt^{\sink}(G,V);\Q)$. We therefore obtain the following from Theorem \ref{genrepstab}.\\

\begin{corollary}
Assume that $|V| = d$. For any $i$, The complex $\Sn_n$-representations $H_i(\Conf_n^{\sink}(G,V);\Q)$ enjoy the following properties:
\begin{enumerate}
\item for any partition $\lambda$, and any integers $n_1 \geq \ldots \geq n_d \geq |\lambda| + \lambda_1$, let $c_{\lambda,n_1,\ldots,n_d}$ be the multiplicity of $S(\lambda)_{n_1,\ldots,n_d}$ in $H_i(\Conf_{\sum_i n_i - (d-1)|\lambda|}^{\sink}(G,V);\Q)$ (see Definition \ref{padpar}). Then the quantity $c_{\lambda,n_1+l,\ldots,n_d+l}$ is independent of $l$ for $l \gg 0$;
\item if $c_{\lambda,n}$ is the multiplicity of $S(\lambda)_n$ in $H_i(\Conf_n^{\sink}(G,V))$, then there exists a polynomial $p(x) \in \Q[x]$ of degree $\leq d-1$ such that for all $n \gg 0$, $c_{\lambda,n} = p(n)$;
\item there is a finite integer $b_i^{(G,V,E)}$, independent of $n$, such that for any irreducible representation $S^{\lambda}$ appearing in $H_i(\Conf_n^{\sink}(G,V))$, $\lambda$ has at most $b_i^{(G,V,E)}$ boxes below its $d$-th row.\\
\end{enumerate}
\end{corollary}

For the remainder of this section we once again assume that $(G,V,E)$ is a tree. Using the results of Section \ref{treecaseI}, we will be able to say a bit more about the asymptotic properties of the $\FI_V$-modules $H_i(\Conf_\dt^{\sink}(G,V))$.\\

To begin, we have the following recursive presentation of the homology groups $H_i(\Conf_\dt^{\sink}(G,V))$.\\

\begin{theorem}\label{recursive}
Let $(G,V,E)$ be a tree, and write $\gamma_{i,n}$ for the rank of $H_i(\Conf_n^{\sink}(G_i,V_i))$, where $(G_i,V_i,E_i)$ is any tree with $|E_i| = i$. Then,
\[
\rank_\Z(H_i(\Conf_n^{\sink}(G,V))) = \begin{cases} \binom{|E|}{i}\gamma_{i,n} &\text{ if $i < |E|$}\\ (-1)^{|E|}(\chi^{(G,V,E)}_n - \sum_{j = 0}^{|E|-1} \binom{|E|}{j}\gamma_{j,n}) &\text{ if $i = |E|$,}\end{cases}
\]
where $\chi^{(G,V,E)}_n$ is the Euler characteristic of $\Conf_n^{\sink}(G,V,E)$.\\
\end{theorem}

\begin{proof}
This follows from Theorem \ref{trivdiff}, the third part of Corollary \ref{bigtreecor}, as well as the fact that the Euler characteristic of a space can be computed as the alternating sum of its Betti numbers.\\
\end{proof}

Using Corollary \ref{eulerchar}, the above formulation allows one to compute the functions $\gamma_{i,n}$ recursively in $i$. These computations appear to get rather complicated as $i$ grows, and we do not know if there is a closed formula. While we do not provide a complete closed form for $\gamma_{i,n}$, there are certain things we can say.\\

\begin{theorem}\label{proper}
Let $\gamma_{i,n}$ be as in Theorem \ref{recursive}. We may write, in accordance with Corollary \ref{fivcor},
\[
\gamma_{i,n} = p_1^{(i)}(n) + \ldots + p_{i+1}^{(i)}(n)(i+1)^n
\]
Let $a_{i,j}$ be the constant term of the polynomial $p_{j+1}^{(i)}$. Then, for all $0 \leq j \leq i$,
\[
a_{i,j} = (-1)^{j}\binom{i}{j}.
\]
In particular, for any tree $(G,V,E)$, and any $0 \leq i \leq |E|$, the polynomials $p_{j+1}^{(i)}$ of Corollary \ref{fivcor} are non-zero whenever $j \leq i$.\\
\end{theorem}

\begin{proof}
We will prove the claim by induction on $i$. If $i = 0$, then $\gamma_{0,n}$ is the $0$-th Betti number of a point. Thus,
\[
\gamma_{0,n} = 1 = (-1)^{0}\binom{0}{0}.
\]
Assume that $a_{i,j} = (-1)^{j}\binom{i}{j}$ for all $j \leq i$, and consider $\gamma_{i+1,n}$. Theorem \ref{recursive} and Corollary \ref{eulerchar} immediately show that $a_{i+1,i+1} = (-1)^{i+1}$. Otherwise, induction and Theorem \ref{recursive} imply that
\[
a_{i+1,j} = (-1)^{j+i+2}\sum_{l=j}^{i}(-1)^{l}\binom{i+1}{l}\binom{l}{j}
\]
However, $\sum_{l=j}^{i}(-1)^{l}\binom{i+1}{l}\binom{l}{j}$ can be shown to be equal to $(-1)^{i}\binom{i+1}{j}$ via a simple induction argument. It follows that
\[
a_{i+1,j} = (-1)^{j+i+2}\sum_{l=j}^{i}(-1)^{l}\binom{i+1}{l}\binom{l}{j} = (-1)^{j}\binom{i+1}{j}.
\]
This completes the proof.\\
\end{proof}

We note that the natural forgetful map $\Phi:\FI_V \rightarrow \FI$, given by forgetting the map into $V$, induces a functor $\Phi_\as:\FI\Mod \rightarrow \FI_V\Mod$. In this way, any $\FI$-module can be considered as an $\FI_V$-module. $\FI_V$-modules which arise in this way tend to be a bit disappointing, as the full richness of the category $\FI_V$ isn't necessary to study them. However, most of the $\FI_V$-modules which have thus far been in the literature have either arisen from $\FI$-modules, or been isomorphic to free $\FI_V$-modules. For this reason, it has been a question of great interest to construct natural examples of $\FI_V$-modules which are ``full'' in some sense. Theorem \ref{proper} implies that the homologies of sink configuration spaces of trees provides a collection of such examples. In fact, it is the belief of the author that the homologies of sink configuration spaces of general graphs will also provide examples of full $\FI_V$-modules.\\


\newpage

\begin{thebibliography}{aaaa}
\small
\bibitem[A]{A} A. Abrams, \textit{Configuration spaces and braid groups of graphs}, Ph.D thesis, \url{home.wlu.edu/~abramsa/publications/thesis.ps}.
\bibitem[BF]{BF} K. Barnett and M. Farber, \textit{Topology of configuration space of two particles on a graph, I}, Algebr.
Geom. Topol. 9(1) (2009), 593–624. \arXiv{0903.2180}.
\bibitem[Cha]{Cha} R. Charney, \textit{An introduction to right-angled Artin groups}, R. Geom Dedicata (2007) 125-141, \url{http://people.brandeis.edu/~charney/papers/RAAGfinal.pdf}.
\bibitem[Che]{Che} S. Chettih, \textit{Dancing in the stars: topology of non-$K$-equal configuration spaces of graphs}, Ph.D. Thesis, University of Oregon, 2016.
\bibitem[Chu]{Chu} T. Church, \textit{Homological stability for configuration spaces of manifolds}, 33 pages Inventiones Mathematicae 188 (2012) 2, 465–504, \arXiv{1103.2441}.
\bibitem[CEF]{CEF} T. Church, J.\,S. Ellenberg and B. Farb, \textit{$\FI$-modules and stability for representations of symmetric groups}, Duke Math. J. 164, no. 9 (2015), 1833-1910.
\bibitem[CEFN]{CEFN} T. Church, J.\,S. Ellenberg, B. Farb, and R. Nagpal, \textit{$\FI$-modules over Noetherian rings}, Geom. Topol. 18 (2014) 2951-2984.
\bibitem[CF]{CF} T. Church and B. Farb, \textit{Representation theory and homological stability}, Advances in Mathematics, (2013), 250-314.
\bibitem[CL]{CL} S. Chettih and D. L\"utgehetmann, \textit{The Homology of Configuration Spaces of Graphs}, \arXiv{1612.08290}.
\bibitem[CP]{CP} F. Cohen and J. Pakianathan, \textit{Configuration spaces and braid groups}, Course Notes, \url{http://web.math.rochester.edu/people/faculty/jonpak/newbraid.pdf}.
\bibitem[EW-G]{EW-G} J. S. Ellenberg and J. D. Wiltshire-Gordon, \textit{Algebraic structures on cohomology of configuration spaces of manifolds with flows}, \arXiv{1508.02430}.
\bibitem[Fa]{Fa} D. Farley, \textit{Homology of tree braid groups}, Topological and asymptotic aspects of group theory, 101-112, Contemp. Math., 394, Amer. Math. Soc., Providence, RI, 2006. \url{http://www.users.miamioh.edu/farleyds/grghom.pdf}.
\bibitem[Far]{Far} M. Farber, \textit{Invitation to Topological Robotics}, Zurich Lectures in Advanced Mathematics, Amer Mathematical Society, 2008.
\bibitem[FH]{FH} E. Fadell and S. Husseini, \textit{Geometry and Topology of Configuration Spaces}, Springer Monographs in Mathematics, Springer-Verlag Berlin Heidelberg, 2001.
\bibitem[Fo1]{Fo1} R. Forman, \textit{Morse theory for cell complexes}, Adv. in Math. 134 (1998), pp. 90-145.
\bibitem[Fo2]{Fo2} R. Forman, \textit{A user's guide to discrete Morse theory}, S\'eminaire Lotharingien de Combinatoire 48 (2002), 35 p. \url{http://www.emis.de/journals/SLC/wpapers/s48forman.pdf}.
\bibitem[FS]{FS} D. Farley and L. Sabalka, \textit{Discrete Morse theory and graph braid groups}, Algebr. Geom. Topol. 5 (2005), 1075-1109 (electronic). \url{http://www.users.miamioh.edu/farleyds/FS1.pdf}.
\bibitem[G]{G} R. Ghrist, \textit{Configuration spaces and braid groups on graphs in robotics}, Knots, braids, and mapping class groups - papers dedicated to Joan S. Birman (New York, 1998), AMS/IP Stud. Adv. Math., 24, Amer. Math. Soc., Providence, RI (2001), 29–40. \url{https://www.math.upenn.edu/~ghrist/preprints/birman.pdf}.
\bibitem[HR]{HR} P. Hersh and V. Reiner, \textit{Representation stability for cohomology of configuration spaces in $\R^d$}, \arXiv{1505.04196}.
\bibitem[KKP]{KKP} J. H. Kim, K. H. Ko, and H. W. Park, \textit{Graph braid groups and right-angled Artin groups}, Trans. Amer. Math. Soc. 364 (2012), 309-360. \arXiv{0805.0082}.
\bibitem[KP]{KP} K. H. Ko, and H. W. Park, \textit{Characteristics of graph braid groups}, Discrete Comput Geom (2012) 48: 915. \arXiv{1101.2648}.
\bibitem[L1]{L1} D. L\"utgehetmann, \textit{Configuration spaaces of graphs}, Masters Thesis,
 \url{http://luetge.userpage.fu-berlin.de/pdfs/masters-thesis-luetgehetmann.pdf}.
\bibitem[L2]{L2} D. L\"utgehetmann, \textit{Representation Stability for Configuration Spaces of Graphs}, \arXiv{1701.03490}.
\bibitem[M]{M} M. Maguire with an appendix by M. Christie and D. Francour, \textit{Computing cohomology of configuration spaces}, \arXiv{1612.06314}.
\bibitem[MW]{MW} J. Miller and J. Wilson, \textit{Higher order representation stability and ordered configuration spaces of manifolds
}, \arXiv{1611.01920}.
\bibitem[R1]{R1} E. Ramos, \textit{Stability phenomena in the homology of tree braid groups}, \arXiv{1609.05611}.
\bibitem[R2]{R2} E. Ramos, \textit{Generalized representation stability and $\FI_d$-modules}, Proc. Amer. Math. Soc., to appear, \arXiv{1606.02673}.
\bibitem[Sa1]{Sa1} S. Sam, \textit{Syzygies of bounded rank symmetric tensors are generated in bounded degree}, Math. Ann., to appear. \arXiv{1608.01722}.
\bibitem[Sa2]{Sa2} S. Sam, Steven V Sam, \textit{Ideals of bounded rank symmetric tensors are generated in bounded degree}, Invent. Math. 207 (2017), no. 1, 1–21, \arXiv{1510.04904}.
\bibitem[Sn]{Sn} A. Snowden, \textit{Syzygies of Segre embeddings and $\Delta$-modules}, Duke Math. J. 162 (2013), no. 2, 225-277, \arXiv{1006.5248}.
\bibitem[SS]{SS} S. Sam, and A. Snowden, \textit{Gr\"obner methods for representations of combinatorial categories}, J. Amer. Math. Soc. 30 (2017), 159-203. \arXiv{1409.1670}.
\bibitem[SS2]{SS2} S. Sam, and A. Snowden, \textit{$GL$-equivariant modules over polynomial rings in infinitely many variables. II}, \arXiv{1703.04516}.

\end{thebibliography}
\end{document}